\documentclass[11pt]{article}

\usepackage[margin=1in]{geometry}
\usepackage{amsmath,amssymb,mathtools}
\usepackage{booktabs,tabularx,array}
\usepackage{enumitem}
\usepackage{xcolor}
\usepackage{microtype}
\usepackage[authoryear,round]{natbib}
\usepackage{hyperref}
\usepackage[nameinlink,capitalise]{cleveref}
\usepackage{amsthm}

\newtheorem{theorem}{Theorem}[section]
\newtheorem{lemma}[theorem]{Lemma}

\theoremstyle{definition}

\theoremstyle{remark}
\newtheorem{remark}[theorem]{Remark}

\hypersetup{
  colorlinks=true,
  linkcolor=blue!55!black,
  citecolor=blue!55!black,
  urlcolor=blue!55!black,
  pdftitle={A New Impossibility Region for the 5 x 5 Symmetric Nonnegative Inverse Eigenvalue Problem},
  pdfauthor={}
}

\usepackage{setspace}
\newcommand{\tr}{\operatorname{tr}}
\newcommand{\Spec}{\operatorname{Spec}}
\newcommand{\diag}{\operatorname{diag}}

\newcommand{\R}{\mathcal R}
\newcommand{\Wreg}{\mathcal W}
\newcommand{\T}{^{\mathsf T}}
\newcommand{\eps}{\varepsilon}
\title{\textbf{A New Impossibility Region for the $5\times5$ Symmetric Nonnegative Inverse Eigenvalue Problem}}
\author{Jiashun Jin,\footnote{For correspondence, send emails to jiashunmail@gmail.com or zke@fas.harvard.edu} $\quad$  Zheng Tracy Ke,  $\quad$ Bingcheng Sui}
\date{\today}

\begin{document}
\maketitle

\begin{abstract}
We present a new impossibility region for the $5\times5$ symmetric
nonnegative inverse eigenvalue problem. The region lies in the low-trace
regime and, to the best of our knowledge, has not been identified previously.
The proof uses a suitable diagonal shift to transform the problem to a
critical high-trace boundary and then reduces a complementary commuting
matrix to a weighted five-cycle. The characteristic polynomial and
spectral identities of this five-cycle provide the main tools for deriving
the resulting contradiction.
\end{abstract}

\section{Introduction}
The nonnegative inverse eigenvalue problem (NIEP) is one of the longstanding
and particularly challenging open problems in matrix theory. For example, the
survey of \citet{JohnsonMarijuanPaparellaPisonero2018} describes the NIEP as a
famously difficult problem in matrix analysis. The origins of the problem can
be traced back to Kolmogorov, who in 1937 asked which complex numbers can occur
as eigenvalues of nonnegative matrices. Suleimanova subsequently extended this
question in 1949 to prescribed lists of eigenvalues, leading to what is now
known as the NIEP \citep{Kolmogorov1937,Suleimanova1949}.
More precisely, let
\[
    \lambda=(\lambda_1,\lambda_2,\ldots,\lambda_n) 
\]
be a list of $n$ complex numbers, counted with multiplicity. The NIEP asks for
necessary and sufficient conditions on $\lambda$ under which there exists an
$n\times n$ entrywise nonnegative matrix $A$ whose spectrum is precisely
$\lambda$. When such a matrix exists, the list $\lambda$ is said to be
\emph{realizable}, and $A$ is called a \emph{realizing matrix}. Throughout
this paper, we refer to $n$ as the \emph{order} of the problem.

An important subproblem arises when all prescribed eigenvalues are real.
Restricting the NIEP to real lists gives the \emph{real nonnegative inverse
eigenvalue problem} (RNIEP). If, in addition, the realizing matrix is required
to be symmetric, one obtains the \emph{symmetric nonnegative inverse
eigenvalue problem} (SNIEP). Another closely related problem is the \emph{symmetric doubly stochastic
inverse eigenvalue problem} (SDIEP), which asks which real spectral lists
can be realized by symmetric doubly stochastic matrices. For $n\leq 4$, the RNIEP and SNIEP are
equivalent: every real spectrum that is realizable by a nonnegative matrix
also admits a symmetric nonnegative realization. This equivalence fails
beginning at order five; in particular, there exist real spectra of size
$n\geq 5$ that are realizable by nonnegative matrices but not by symmetric
nonnegative matrices
\citep{JohnsonLaffeyLoewy1996,EglestonLenkerNarayan2004}.

The NIEP, RNIEP, and SNIEP are all completely solved for orders $n\leq 4$.
Beginning at $n=5$, however, none of the three problems has been completely
solved in general, although many important special cases and spectral regions
have been characterized; see
\citep{EglestonLenkerNarayan2004,JohnsonMarijuanPaparellaPisonero2018}.
Thus, order five is the first genuinely open case for the SNIEP.

In this paper, we focus on the SNIEP of order five. The order-five problem has
received considerable attention over the past several decades. One important
constructive approach is due to Soules \citep{Soules1983}, who introduced a
class of orthogonal eigenvector matrices that can be used to construct
symmetric nonnegative matrices with prescribed spectra. McDonald and Neumann
\citep{McDonaldNeumann2000} studied Soules' approach in detail for matrices of
order at most five. Loewy and McDonald \citep{LoewyMcDonald2004} subsequently
developed a systematic analysis of the $5\times5$ SNIEP and derived several
important structural and spectral restrictions.

Further progress was obtained by Spector, who gave a complete characterization
of the trace-zero $5\times5$ SNIEP \citep{Spector2011}. Spector later
constructed symmetric nonnegative families realizing a previously unknown
part of the order-five region \citep{Spector2014}. Loewy and Spector then
established an important high-trace result \citep{LoewySpector2017}. More
precisely, let $\lambda_1\geq\lambda_2\geq\cdots\geq\lambda_5$ 
be the eigenvalues of a symmetric entrywise-nonnegative $5\times5$ matrix
$A$, and write $S=\sum_{i=1}^5\lambda_i=\tr(A)$. 
They proved, in particular, that if $\lambda$ is realizable, then 
\begin{equation} \label{eq:hightrace} 
S\geq \frac{\lambda_1}{2}
\qquad\Longrightarrow\qquad
\lambda_3\leq S.
\end{equation} 
Motivated by this threshold, we use the term \emph{high-trace regime} for
$S\geq\lambda_1/2$ and \emph{low-trace regime} for $S<\lambda_1/2$. These
terms are used here only as convenient terminology for the order-five problem.
The high-trace regime is substantially better understood, and  the
low-trace regime contains the more delicate cases and is less understood.  

In recent years, the low-trace region has been studied from several directions. On the
necessary-condition side,
\citet{JohnsonMarijuanPisonero2017,JohnsonMarijuanPisonero2021} developed
interlacing arguments that rule out many five-spectra, while \citet{Loewy2021}
derived a stronger moment-type necessary condition for order five. On the
sufficient-condition side, a broad collection of realizability criteria has
also been developed; see, for example,
\citep{SotoRojoMoroBorobia2007,EllardSmigoc2016,MarijuanPisoneroSoto2017}.
Despite this substantial progress, the order-five SNIEP is not yet completely
resolved.

Our contribution is on the necessary-condition side.  Consider the set of all possible $\lambda 
= (\lambda_1, \lambda_2, \ldots, \lambda_5)$.  After removing the cases settled by existing results, the remaining cases are contained in a region denoted by $\R$, defined precisely in \eqref{eq:R} below. 
In particular,
lists in $\R$ have three positive and two negative eigenvalues and satisfy
\[ 
0<S<\min\left\{\lambda_3,\frac{\lambda_1}{2}\right\}; 
\] 
see (\ref{eq:R}) for details. 
Our main result identifies a subregion of $\R$ defined by 
\[
\Wreg=\left\{\lambda\in\R:4\lambda_1\leq 9\lambda_3-S\right\},
\]
and we show that no list in $\Wreg$
admits a symmetric  nonnegative realization. To the best of our knowledge, this
explicit impossibility region has not previously been identified.

We briefly describe the main mechanism of the proof. Suppose, for
contradiction, that a list in $\Wreg$ has a symmetric nonnegative realization
$A$. Since $S<\lambda_3$, the diagonal shift
\[
c=\frac{\lambda_3-S}{4}>0,
\qquad
B=A+cI_5,
\]
places the shifted matrix at a critical trace level: the third eigenvalue of
$B$ is exactly $\tr(B)$, while the defining inequality of $\Wreg$ ensures
$\tr(B)\geq\lambda_1(B)/2$. This places $B$ in the high-trace regime (but with a shifted spectrum); see (\ref{eq:hightrace}).  At the same time,  we show that there exists  a $5 \times 5$ symmetric nonnegative matrix $X$ which commutes with $B$ ($B X = XB$) and satisfies some additional requirements,   and we show that the positive off-diagonal support graph of such an $X$ must be a five-cycle. 

Next, we divide the proof into three cases: (I)  repeated-positive, 
(II)  nonrepeated strict, and (III) nonrepeated equality boundary.  
In (I) and (III), the commutation relation further yields a representation
of the form
\begin{equation}\label{eq:B}
B=dI_5+tX^2,
\end{equation}
for suitable scalars $d$ and $t>0$.  The characteristic polynomial and spectral
identities of a positive weighted five-cycle then impose strong restrictions
on the eigenvalues of $X$, leading to a contradiction. In case (II),  we have a 
 positive margin in the high-trace inequality, so we can prove the claim 
 using a perturbation approach.  Therefore,  the three cases share the same shift-and-cycle
reduction,  though the final step is different. 

Our construction of $X$ is related to the complementary commuting matrix construction of Laffey \citep{Laffey1998}, while our perturbation and strictness
argument at the high-trace boundary is related to the boundary arguments
used in the order-five SNIEP by \citet{LoewySpector2017}.
The main technical distinction of our proof is the way these ideas are
combined with the specific diagonal shift above.  This shift moves the
original low-trace problem to the critical level
$\lambda_3(B)=\tr(B)$, 
while retaining the high-trace inequality
$\tr(B)\geq \frac{\lambda_1(B)}{2}$. 
We then show that the positive off-diagonal support of this complementary commuting matrix is a weighted
five-cycle.  This reduction makes the characteristic polynomial and
spectral identities of the five-cycle the main tools in the final argument.

So far, our discussion has been mainly theoretical. On the application side,
the NIEP and closely related nonnegative matrix factorization (NMF)  problems have
also received attention in network modeling. For example,
\citet{Jin2022NMF} developed sharp results for nonnegative matrix
factorization and used them to study when general low-rank network models
can be represented by interpretable mixed-membership models.
More recently, \citet{JinHuang2024} studied the NIEP and NMF jointly,
developing constructive results with further implications for social network
modeling. We discuss these connections in more detail in
Section~\ref{sec:network}.

We finish this section by fixing notation used throughout the paper. Unless
stated otherwise, the eigenvalues are ordered as
\[
\lambda_1\geq\lambda_2\geq\lambda_3\geq\lambda_4\geq\lambda_5,
\qquad
S=\sum_{i=1}^5\lambda_i.
\]
We write $I_5$ for the $5\times5$ identity matrix, $\tr(M)$ for the trace of a
matrix $M$, and $\Spec(M)$ for its spectrum, counted with multiplicity. For a
symmetric nonnegative matrix $M$, its \emph{positive off-diagonal support
graph} has vertex set $\{1,\ldots,5\}$ and an edge $\{i,j\}$ exactly when
$M_{ij}>0$ for $i\neq j$. We write $C_5$ for the cycle on five vertices and
interpret cycle indices modulo $5$ when no confusion can arise. The residual
region and the new impossibility region are denoted by $\R$ and $\Wreg$,
respectively, and are defined in Section~2. The letter $A$ will generally
denote a hypothetical realizing matrix, while $B$ denotes its diagonal shift
and $X$ the complementary commuting matrix arising in the proof.

The remainder of the paper is organized as follows. Section~2 states the main
result and explains why the analysis can be reduced to $\R$. Section~3
discusses the connection with network modeling. Section~\ref{sec:proof} gives
the proof of the main theorem, beginning with three structural lemmas and then
treating the three spectral cases separately.  The concluding part of the paper contains a short discussion, acknowledgements, and the proofs of the structural lemmas used in the main argument.

\section{Main result} \label{sec:main} 
Throughout this paper, let
\[
\lambda=(\lambda_1,\lambda_2,\ldots,\lambda_5),
\qquad
\lambda_1\geq\lambda_2\geq\lambda_3\geq\lambda_4\geq\lambda_5,
\qquad
S=\sum_{i=1}^5\lambda_i,
\]
where all $\lambda_i$ are real. We study when $\lambda$ is the spectrum of a
$5\times5$ symmetric entrywise-nonnegative matrix.
As summarized by \citet{MarijuanPisonero2018}, after the previously settled
order-five cases are removed, it is enough to consider spectra satisfying
\[
\lambda_1=\max_i|\lambda_i|,
\qquad
\lambda_1>\lambda_2\geq\lambda_3>0>\lambda_4\geq\lambda_5,
\qquad
0<S<\lambda_3.
\]
Moreover, \citet{LoewySpector2017} showed that if $\lambda$ is realizable and
$S\geq\lambda_1/2$, then necessarily $\lambda_3\leq S$. Combining this with
$S<\lambda_3$, we may further restrict attention to $S<\lambda_1/2$.
Consequently, the cases not covered by these existing results are contained
in
\begin{equation}\label{eq:R}
\R=
\left\{
\lambda:
\lambda_1=\max_i|\lambda_i|,
\ \lambda_1>\lambda_2\geq\lambda_3>0>\lambda_4\geq\lambda_5,
\ 0<S<\min\{\lambda_3,\lambda_1/2\}
\right\},
\end{equation}
which we call the \emph{residual region}. See
Remark~\ref{rem:whyR} for a more detailed explanation.

We now focus on $\R$ and introduce the subregion
\begin{equation}\label{eq:W}
\Wreg=
\left\{
\lambda\in\R:
4\lambda_1\leq 9\lambda_3-S
\right\}.
\end{equation}
The set is nonempty (see Remark 2.4).  Our main result is the following theorem.

\begin{theorem}[A new impossibility region]\label{thm:main}
No list in $\Wreg$ is the spectrum of a real symmetric
entrywise-nonnegative $5\times5$ matrix.
\end{theorem}
Theorem~\ref{thm:main} removes the entire closed side
$4\lambda_1\leq 9\lambda_3-S$ 
from the residual region. Therefore, any list in $\R$ that remains
potentially realizable must lie in
\[
\R\setminus\Wreg
=
\left\{
\lambda\in\R:
4\lambda_1>9\lambda_3-S
\right\}.
\]
To the best of our knowledge, the impossibility region identified in
Theorem~\ref{thm:main} has not been identified previously.

\begin{remark}[Why it is enough to focus on $\R$]\label{rem:whyR}
We explain why the cases not covered by earlier results are contained
in $\R$. The standard necessary conditions for a nonnegative realization give
\[
\lambda_1=\max_i|\lambda_i|,
\qquad
S=\tr(A)\geq0.
\]
The order-five trace-zero case is completely characterized by
\citet{Spector2011}, and the remaining lower-order and previously classified
sign patterns are covered by the order-four theory and the classical
order-five results; see, for example,
\citep{LoewyLondon1978,Fiedler1974,LoewyMcDonald2004,
Spector2014,MarijuanPisonero2018}.
After these cases are removed, the remaining positive-trace configuration
can be restricted to
\begin{equation}\label{eq:remaining-sign}
\lambda_1=\max_i|\lambda_i|,
\qquad
\lambda_1>\lambda_2\geq\lambda_3>0>\lambda_4\geq\lambda_5,
\qquad
S>0.
\end{equation}
There are two further reductions. First, suppose that $S\geq\lambda_3$.
Then
\[
\lambda_1+\lambda_2+\lambda_4+\lambda_5
=
S-\lambda_3
\geq0.
\]
The four-list $(\lambda_1,\lambda_2,\lambda_4,\lambda_5)$ also satisfies the
Perron condition $\lambda_1=\max_i|\lambda_i|$. By the order-four theory
\citep{LoewyLondon1978,Fiedler1974}, there exists a symmetric nonnegative
$4\times4$ matrix $B$ with spectrum
$(\lambda_1,\lambda_2,\lambda_4,\lambda_5)$. Since $\lambda_3>0$, the block
matrix
\[
A=
\begin{pmatrix}
B&0\\
0&\lambda_3
\end{pmatrix}
\]
is symmetric and nonnegative and has spectrum $\lambda$. Hence the range
$S\geq\lambda_3$ is already realizable. Therefore, any case not covered by
the preceding results must satisfy
\[
0<S<\lambda_3.
\]
Second, suppose that $S\geq\lambda_1/2$. The high-trace theorem of
\citet{LoewySpector2017} implies that realizability would require
\[
\lambda_3\leq S.
\]
This contradicts $S<\lambda_3$. Thus the portion
\[
0<S<\lambda_3,
\qquad
S\geq\frac{\lambda_1}{2},
\]
is already known to be nonrealizable. We may therefore further restrict to
\[
S<\frac{\lambda_1}{2}.
\]
Combining these reductions with \eqref{eq:remaining-sign} gives exactly the
conditions defining $\R$ in \eqref{eq:R}. Thus the cases not covered by the
existing results described above are contained in the residual region $\R$.
\end{remark}

\begin{remark}[Why the boundary has this form]\label{rem:whyboundary} 
We explain why the boundary that  defines the new impossibility region has the form of 
$4 \lambda_1 \leq 9 \lambda_3 - S$.  
Suppose $\lambda\in\Wreg$ is realizable with a realizing matrix $A$.
Previous work suggests that the high-trace regime
\[
S\geq \frac{\lambda_1}{2}
\]
may be more tractable than the low-trace regime (e.g., 
\cite{LoewySpector2017}).  However, since $\Wreg\subset\R$, we have
$S<\lambda_1/2$, so $A$ lies in the low-trace regime.  This motivates us
to construct a new symmetric nonnegative matrix $B$ that lies in the
high-trace regime. For $c>0$, consider the diagonal shift
\[
B=A+cI_5.
\]
In order for $B$ to be in the high-trace regime, we must have $\tr(B) \geq \lambda_1(B) / 2$, or equivalently 
\begin{equation}\label{eq:c-high}
c\geq \frac{\lambda_1-2S}{9}.
\end{equation}
At the same time,  the high-trace theorem \citep{LoewySpector2017} states that if a symmetric
nonnegative matrix $M$ realizes an ordered spectrum and
$\tr(M)\geq \frac{\lambda_1(M)}{2}$, then we must have 
$\lambda_3(M)\leq \tr(M)$. 
Therefore, if $B$ lies in the high-trace regime, then we must have 
$\lambda_3(B)\leq\tr(B)$, or  equivalently,  
\[
c\geq\frac{\lambda_3-S}{4}.
\]  
Here, for our proof below (e.g., Lemmas \ref{lem:positive-strict}-\ref{lem:certificate}), it is especially useful to consider the equality case where we choose $c$ such that 
$\lambda_3(B)=\tr(B)$.   Merely placing $B$ in the high-trace regime does
not by itself provide enough structure for our argument.  At equality,
however, $B$ lies on the boundary of the high-trace necessary condition.
As shown in Lemma~\ref{lem:positive-strict}, a strictly positive matrix in
the high-trace regime must satisfy the strict inequality
\[
\lambda_3(B)<\tr(B).
\]
Moreover, the boundary condition is precisely what is used in
Lemma~\ref{lem:certificate} to obtain a complementary commuting
matrix, whose support is eventually reduced to a five-cycle.
This motivates the critical choice where we choose $c$ in $B$ by 
\[
c = c_*=\frac{\lambda_3-S}{4}. 
\]
Now,  by (\ref{eq:c-high}),  in order for $B$ in the high-trace regime,   we must have 
\[
c_* \geq  
\frac{\lambda_1-2S}{9}, 
\qquad \mbox{or equivalently}, \qquad 
4\lambda_1\leq 9\lambda_3-S.
\]
This is precisely the inequality defining $\Wreg$. 
\end{remark}

\begin{remark}[The region $\Wreg$ is nonempty]\label{rem:Wnonempty}
The region in Theorem~\ref{thm:main} is genuinely nonempty. For example,
consider
\[
\lambda=
\left(
1,\frac{1}{2},\frac{1}{2},-\frac{19}{24},-\frac{19}{24}
\right).
\]
It is seen that $\lambda_1=\max_i|\lambda_i|=1$,  $S = 5 / 12$, $\min\{\lambda_1/2, \lambda_3\} = 1/2$, 
and so $0 < S < \min \{\lambda_3,\frac{\lambda_1}{2}\}$.  Therefore,  $\lambda\in\R$. Moreover,
$4\lambda_1 - (9 \lambda_3 - S)  = - 1/12 < 0$ so $4 \lambda_1 < 9  \lambda_3 - S$.  
Therefore,   $\lambda\in\Wreg$, and Theorem~\ref{thm:main} implies that this
list does not admit a symmetric nonnegative realization.
\end{remark}

\section{Connections to network modeling} \label{sec:network}   
So far, we have been focused on the theoretical side. On the application side, SNIEP 
can be useful in problems such as network modeling.  In this section, we  discuss how SNIEP and network modeling are connected. 
 Consider an undirected network with $n$ nodes and $K$ communities (communities are tightly woven nodes that have more edges within than between \citep{JinKeTangWang2025}). Let $A$ be the adjacency matrix of the network and let $\Omega$ be the corresponding Bernoulli probability matrix, where  
\[
\Omega_{ij}=\mathbb{P}(A_{ij}=1),   \mbox{for all $i \neq j$},  \mbox{and so $\mathbb{E}[A] = \Omega - \mathrm{diag}(\Omega)$}. 
\]
We call this a rank-$K$ network model when $\operatorname{rank}(\Omega)=K$.

Among all rank-$K$ network models, the {\it degree-corrected mixed-membership (DCMM)} model 
fits well with real networks and is especially useful 
 \citep{JinKeTangWang2025}, which writes 
\[
\Omega=\Theta\Pi P\Pi'\Theta.
\]
Here $\Theta=\operatorname{diag}(\theta_1,\ldots,\theta_n)$ models degree heterogeneity, each row of $\Pi$ is a probability vector describing the mixed membership of a node, and the symmetric entrywise-nonnegative matrix $P\in\mathbb{R}^{K\times K}$ describes baseline connectivity between communities. The DCMM includes  the popular  stochastic block model (SBM) \citep{HollandLaskeyLeinhardt1983} and the degree-corrected block model (DCBM) \citep{KarrerNewman2011} 
as special cases.  

We now compare the DCMM model and the rank-$K$ model.  The DCMM is more useful: all parameter matrices $(\Theta, \Pi, P)$ have practical meanings, which makes
the model more interpretable and more useful for community detection, mixed-membership estimation, and related network-analysis tasks \citep{KeJin2023,JinKeLuo2024}. The rank-$K$ formulation is broader but carries less explicit community structure.   
Naturally, we wish to study how broad the DCMM model is. This motivates the following problem in network modeling: 
\begin{equation} \label{networkmodeling} 
\mbox{When can we represent a  general rank-$K$ network model in the DCMM form?}. 
\end{equation} 
Seemingly, the problem is essentially a non-negative matrix factorization (NMF)  problem. But what is 
interesting is that, the problem is also related to SNIEP, or the {\it symmetric doubly stochastic inverse eigenvalue problem (SDIEP)} 
to be precise.  
In detail, consider a rank-$K$ network model and suppose that $\Omega$ has exactly $m$ negative nonzero eigenvalues, where we note $0\le m\le K-1$. Let 
$a_{K,m}=\max\{0,2m-K\}$ and $\rho_{K,m}=1+a_{K,m}$. 
Define
\[
\sigma_{K,m}
=
\left(
\rho_{K,m},
\underbrace{1,\ldots,1}_{K-m-1},
\underbrace{-1,\ldots,-1}_{m}
\right).
\]
In \citep{Jin2022NMF}, the author pointed out in order to solve the network modeling problem in (\ref{networkmodeling}), 
a key component is to solve the SDIEP problem: 
\[
\mbox{for which $(K, m)$, $\sigma_{K,m}$, up to a scaling factor, admits a symmetric doubly stochastic realization}.  
\] 
The underlying argument is elementary but a little bit long, so we skip the details here (see \citep{Jin2022NMF}, where we have  the details). 
In the case of $m \leq K/2$, $\sigma_{K, m}$ is always realizable and 
 \citep{Jin2022NMF} gives the required realization.  The more difficult case of  $m>K/2$ is studied in \citet{JinHuang2024}, where additional constructions and obstructions are developed.

Such a connection demonstrates that SNIEP can be useful in applications such as network modeling, and motivates 
a new research direction for SDIEP.

\section{Proof of Theorem~\ref{thm:main}}\label{sec:proof} 
We prove by contradiction. Fix a $\lambda = (\lambda_1, \lambda_2, \ldots, \lambda_5) \in \Wreg$ and let $S=\sum_{k=1}^5\lambda_k$. By definition, 
\begin{equation} \label{W-cond}
\lambda_1>\lambda_2\ge \lambda_3>0>\lambda_4\ge\lambda_5, \qquad
0<S<\min\Bigl\{\lambda_3, \frac{\lambda_1}{2}\Bigr\}, 
\qquad
4\lambda_1\le 9\lambda_3-S,
\end{equation}
Suppose $\lambda$ is realizable by 
a $5 \times 5$ nonnegative symmetric matrix $A$. We first construct a shifted matrix $B$ and derive some useful properties of $B$.

\begin{lemma}[The shifted matrix]\label{lem:shiftedB}
Suppose a real symmetric entrywise-nonnegative matrix $A$ has spectrum
$(\lambda_1,\ldots,\lambda_5)$, where  $\lambda$ satisfies  \eqref{W-cond}. Let 
\begin{equation} \label{DefineB}
B=A+cI_5, \qquad \mbox{with}\quad c=\frac{\lambda_3-S}{4}.
\end{equation}
Let $(\theta_1,\ldots,\theta_5)$ denote the spectrum of $B$, and $\tau=\tr(B)$. Then, $B$ is irreducible and has strictly positive diagonal entries. Its spectrum satisfies that 
\begin{equation} \label{B-spectrum}
\theta_1>\theta_2\geq \theta_3=\tau>0>\theta_4\geq \theta_5, \qquad  \tau\geq \frac{\theta_1}{2}. 
\end{equation}
Furthermore, $\tau=\theta_1/2$ if and only if $4\lambda_1=9\lambda_3-S$. 
\end{lemma}

Hence, the shifted matrix $B$ is irreducible, has strictly positive diagonal, lies in the high-trace regime ($\tr(B)\geq \theta_1/2$), and satisfies the boundary identity $\tr(B)=\theta_3$.

The following lemma is proved in the appendix.  It focuses on entrywise strictly positive matrices and refines the Loewy–Spector high-trace condition for general entrywise nonnegative matrices: under strict positivity, the non-strict inequality in the Loewy–Spector condition becomes strict.
%%%%%%%%%%%%%
%%%%%%%%%%%%%
\begin{lemma}[Strictness for a positive matrix]\label{lem:positive-strict}
Let $P$ be a real symmetric entrywise \emph{strictly positive} $5\times5$ matrix with ordered eigenvalues
$\mu_1\ge\mu_2\ge\mu_3\ge\mu_4\ge\mu_5$. If  $\tr(P) \geq  \mu_1/2$, then
$\mu_3 < \tr(P)$. 
\end{lemma}
Since the shifted matrix $B$ satisfies the requirement of Lemma \ref{lem:positive-strict}, it follows that some off-diagonal entries of $B$ must be $0$.  This observation, together with the irreducibility of $B$, ensures that there exists a nonzero symmetric $5 \times 5$ nonnegative matrix $X$ such that 
\[
\mbox{$X_{ij} B_{ij} = 0$ for all $i, j$ and $B X = X B$}  \qquad (\mbox{note: since $B_{ii} > 0$, we must have $X_{ii} = 0$}).   
\]
Since $X_{ij}B_{ij}=0$ implies that the support of $X$ (locations where $X$ is nonzero) is contained in the complement of the support of $B$, we call $X$ a \emph{complementary commuting matrix} for $B$. Here, the strictness in Lemma \ref{lem:positive-strict} is essential. Without the strictness, $B$ may not have any zero entry; consequently, $X$ will become a zero matrix, which is undesirable for our proofs.

We now establish the existence of the complementary commuting matrix. 
For any symmetric nonnegative matrix $X$ that has zero diagonal, its \emph{positive off-diagonal support graph} has vertices $1,\ldots,5$, and vertices $i$ and $j$ are joined by an edge exactly when $X_{ij}>0$.  We write $C_5$ for the cycle on five vertices, $K_2$ for a single edge, and $P_3$ for a path on three vertices.  Thus $K_2\cup P_3$ denotes two disconnected components, one a single edge and the other a three-vertex path. The following lemma is proved in the appendix. 
%%%%%%%%%%
%%%%%%%%%%
\begin{lemma}[Complementary commuting matrix and support reduction]\label{lem:certificate}
Let $B$ be an irreducible real symmetric entrywise-nonnegative $5\times5$ matrix with strictly positive diagonal.  Denote the  ordered eigenvalues of $B$ by  $\theta_1\ge\theta_2\ge\theta_3\ge\theta_4\ge\theta_5$, 
and assume $
\tr(B)= \theta_3  > 0$ and $\tr(B) \geq \theta_1/2$. 
There exists a nonzero real symmetric entrywise-nonnegative matrix $X$ such that
\begin{equation}\label{eq:certificate}
\mbox{$X_{ij} B_{ij} = 0$ for all $1\leq i, j\leq 5$} \qquad  \mbox{and}  \qquad
BX=XB.
\end{equation}
Also, for every such $X$, the positive off-diagonal support graph of $X$ is isomorphic to either
$C_5$ or $K_2\cup P_3$. In particular, the support graph of $X$ is isomorphic to $C_5$ if 
\begin{equation} \label{eq:sign-B}
\theta_1>\theta_2\geq \theta_3 >0>\theta_4\ge\theta_5
\end{equation}
\end{lemma}

For the matrices $B$ arising here, (\ref{eq:sign-B}) is the relevant case. 
Lemma \ref{lem:certificate} shows that the positive off-diagonal support graph of $X$ is isomorphic to the five-cycle $C_5$.  
This significantly narrows down the range of possible $X$ (and hence of possible $A$ and $B$).

The next lemma records the algebraic consequence of the five-cycle support and the commutation relation.  
It provides a more explicit relationship between $B$ and $X$. 

\begin{lemma}[Cycle-commutation representation]\label{lem:cycle-rep}
Let $B$ be an irreducible real symmetric entrywise-nonnegative $5\times5$ matrix with strictly positive diagonal, and let $X$ be a nonzero real symmetric entrywise-nonnegative matrix such that $BX=XB$ and that 
$B_{ij}X_{ij}=0$ for all $1\leq i,j\leq 5$. 
Assume that the positive off-diagonal support graph of $X$ is isomorphic to $C_5$.  Then there exist $a\in\mathbb R$ and $b>0$ such that
\begin{equation}\label{eq:cycle-rep}
B=aI_5+bX^2.
\end{equation}
Moreover, the positive off-diagonal support graphs of $B$ and $X$ are complementary five-cycles.
\end{lemma}

Meanwhile, since $X$ has support $C_5$, its characteristic polynomial has several useful special properties, collected in the following lemma, whose proof is deferred to the appendix.
 
\begin{lemma}[Spectral identities for a positive weighted five-cycle]\label{lem:C5poly}
Let $X$ be a real symmetric $5\times5$ matrix with zero diagonal whose positive off-diagonal support is a five-cycle.  After a cyclic relabeling, write
\[
X=
\begin{pmatrix}
0&y_1&0&0&y_5\\
y_1&0&y_2&0&0\\
0&y_2&0&y_3&0\\
0&0&y_3&0&y_4\\
y_5&0&0&y_4&0
\end{pmatrix},
\qquad y_1,\ldots,y_5>0.
\]
We have the following claims. 
\begin{enumerate}[label=(\alph*),leftmargin=2.2em]
\item The characteristic polynomial of $X$ is
\begin{equation}\label{eq:C5poly}
p_X(t)=\det(tI_5-X)
=t^5-\left(\sum_{i=1}^5y_i^2\right)t^3
+\left(\sum_{i=1}^5y_i^2y_{i+2}^2\right)t
-2\prod_{i=1}^5y_i,
\end{equation}
where the indices are read modulo $5$ (e.g., $y_6 = y_1$, $y_7 = y_2$).
\item $X$ is nonsingular and $\det(X)=2\prod_{i=1}^5y_i>0$.   
Moreover, $X$ has exactly three positive and two negative eigenvalues.

\item $X$ has no pair of opposite nonzero eigenvalues.  That is, $r\ne0$ and $p_X(r)=0$ imply $p_X(-r)\ne0$.
\item Let $\mu_1,\ldots,\mu_5$ denote the eigenvalues of $X$ and define $c_m(\mu)=\sum_{j=1}^5\mu_j^m$. Then, \[
c_1(\mu)=c_3(\mu)=0, \qquad c_2(\mu)= 2\sum_{i=1}^5y_i^2, \qquad \prod_{j=1}^5\mu_j=2\prod_{i=1}^5y_i, \qquad 
\sum_{1\leq i<j<k\leq 5}\mu_i\mu_j\mu_k=0.
\]
\end{enumerate}
\end{lemma}

\medskip
We are ready to prove Theorem~\ref{thm:main}. So far, we have shown that the spectrum of $B$ satisfies \eqref{B-spectrum} and that the spectrum of $X$ satisfies the properties in Lemma~\ref{lem:C5poly}. 
However, the spectra of $B$ and $X$ are explicitly connected through $B=aI_5+bX^2$. 
We will show that under this  identity, \eqref{B-spectrum} and Lemma~\ref{lem:C5poly} cannot hold simultaneously.

Recall that for any $\lambda \in \Wreg$, it satisfies \eqref{W-cond}. 
We split $\lambda \in \Wreg$ into three cases:  (I), (II), and (III). 
\begin{enumerate}[leftmargin=2.2em]
\item  (I). Repeated positive: $\lambda_2=\lambda_3$ and $4\lambda_1\le 9\lambda_3-S$;
\item  (II). Nonrepeated positive strict:  $\lambda_2>\lambda_3$ and $4\lambda_1<9\lambda_3-S$;
\item (III). Nonrepeated positive boundary:  $\lambda_2>\lambda_3$ and $4\lambda_1=9\lambda_3-S$.
\end{enumerate}
For each case, we will separately establish a contradiction between \eqref{B-spectrum} and Lemma~\ref{lem:C5poly}.

\subsection{Case (I): repeated positive eigenvalues}\label{subsec:case-repeated}
In this case, $\lambda_2=\lambda_3$. Hence, for the shifted matrix \(B\) in Lemma~\ref{lem:shiftedB}, it satisfies that
%%%%%%%%%%%%%%%%%
%%%%%%%%%%% 
\begin{equation} \label{B-spectrum-case(I)}
\theta_1>\theta_2= \theta_3=\tau>0>\theta_4\geq \theta_5, \qquad  \tau\geq \frac{\theta_1}{2}. 
\end{equation}
Lemma~\ref{lem:certificate} therefore yields a nonzero complementary commuting
matrix \(X\) whose positive off-diagonal support graph is \(C_5\), and
Lemma~\ref{lem:cycle-rep} gives
\begin{equation}\label{eq:Bpoly-rep-rs}
B=aI_5+bX^2,
\qquad b>0,
\end{equation}
and the off-diagonal support graphs of \(B\) and \(X\) are complementary
five-cycles. Moreover, Lemma~\ref{lem:C5poly} applies to \(X\). In particular, \(X\) is nonsingular,
has exactly three positive and two negative eigenvalues, and has no pair of
opposite nonzero eigenvalues.

Our proof has four steps: In the first step, we show that the spectrum of $X$ can be expressed as $|s_0|\cdot 
(x,y,z,-1,-1)$, for some $s_0\in\mathbb{R}$ and $x,y,z>0$. 
In the second step, we show that the possible range of $x$ is $x\in (1, 6/5)$. In the third step, we use the range of $x$ to show that the eigenvalues of $B$ must satisfy $\theta_5/\theta_3<-2$. 
In the last step, we show that $\theta_5/\theta_3<-2$ contradicts \eqref{B-spectrum-case(I)}.  

\medskip
\noindent
{\bf Step 1: the spectrum of $X$.} Let \(\mu_1\geq \mu_2\geq \ldots\geq \mu_5\) denote the eigenvalues of \(X\). By \eqref{eq:Bpoly-rep-rs}, the eigenvalues of \(B\) are $
\theta_i = a+b\mu_i^2$. Since $\theta_2=\theta_3$ and $b>0$, 
there are distinct indices \(i,j\) such that $
\mu_i^2=\mu_j^2$. 
Since \(X\) is nonsingular, both \(\mu_i\) and \(\mu_j\) are nonzero. By Lemma~\ref{lem:C5poly}, 
\(X\) has no pair of opposite nonzero eigenvalues. Therefore, we must have
\[
\mu_i=\mu_j; \qquad \mbox{and denote this repeated nonzero eigenvalue by \(s_0\)}. 
\]

We show that the sign of $s_0$ can only be negative. 
Suppose otherwise \(s_0>0\). By Lemma~\ref{lem:C5poly}, \(X\) has three positive and two negative
eigenvalues. After dividing its spectrum by \(s_0\) we may write
\[
\frac{\Spec(X)}{s_0}
=
(x,1,1,-y,-z),
\qquad x,y,z>0.
\]
Because the Perron eigenvalue of the irreducible nonnegative matrix \(X\) is
simple, the repeated eigenvalue \(s_0\) cannot be its Perron eigenvalue. This gives $x>1$. 
The identities $
\sum_{i=1}^5\mu_i=0$ and $
\sum_{1\leq i<j<k\leq5}\mu_i\mu_j\mu_k=0$ 
from Lemma~\ref{lem:C5poly} imply that 
\[
y+z=x+2,
\qquad
yz=\frac{2(x+1)^2}{x+2}.
\]
We then construct a quadratic polynomial with roots $y$ and $z$: 
\[
p(t)=t^2-(x+2)t+\frac{2(x+1)^2}{x+2}. 
\]
Note that $p(1)=x(x+1)/(x+2)>0$. It implies that the two roots can only be both larger than $1$ or both smaller than $1$.  Since $y+z=x+2>2$, we conclude that 
\[
y>1, \qquad z>1.
\]
Consequently, $
x^2$, $y^2$ and $z^2$ are all larger than $1$. 
By \eqref{eq:Bpoly-rep-rs}, this would give at least three eigenvalues of
\(B\) strictly larger than $
a+bs_0^2=\theta_2=\theta_3$.   
However, $\theta_2$ is the second largest eigenvalue, so \(\theta_1\) is the only eigenvalue of \(B\) strictly larger than
\(\theta_2\). This is impossible. Hence $s_0<0$. 

So far, we have shown that the two repeated eigenvalues are negative. By Lemma~\ref{lem:C5poly}, $X$ has three positive eigenvalues and two negative eigenvalues. We thereby divide the spectrum of \(X\) by $
|s_0|=-s_0$ 
and write
\begin{equation} \label{X-spectrum}
\frac{\Spec(X)}{|s_0|}
=
(x,y,z,-1,-1),
\qquad x,y,z>0.
\end{equation}
Without loss of generality, let $x|s_0|$ be the Perron eigenvalue of $X$. Then, $x\geq 1$. 
An equality \(x=1\) would give the opposite nonzero eigenvalues $
\pm |s_0|$, 
which contradicts Lemma~\ref{lem:C5poly}. Thus, $x>1$. 
With the parametrization in \eqref{X-spectrum}, the two identities $
\sum_{i=1}^5\mu_i=0$ and $
\sum_{1\leq i<j<k\leq5}\mu_i\mu_j\mu_k=0$ 
from Lemma~\ref{lem:C5poly} become 
\begin{equation}\label{yz-identity}
y+z=2 - x,
\qquad
yz=\frac{2(x-1)^2}{2-x}.
\end{equation}
In particular, $
y+z=2-x<1$. Thus, \(0<y,z<1\). Without loss of generality, suppose $y\geq z$. We have  shown that
\begin{equation}  \label{X-spectrum2}
x^2>1>y^2\geq z^2,
\end{equation}
The spectrum of $X$ is then described by \eqref{X-spectrum}, \eqref{yz-identity}, and \eqref{X-spectrum2}. 

\medskip
\noindent
{\bf Step 2: the range of $x$}. 
We have seen that $x>1$, but we also wish to derive an upper bound for $x$. Note that \eqref{B-spectrum-case(I)} implies
\begin{equation} \label{eq:ratio13<2}
\theta_1/\theta_3\leq 2. 
\end{equation}
Our plan is to express $\theta_1/\theta_3$ as a function of $x$ and then apply \eqref{eq:ratio13<2} to get an upper bound for $x$.

Let $q=b|s_0|^2>0$. By \eqref{X-spectrum} and the identity $B=aI_5+bX^2$,  the eigenvalues of $B$ are 
\begin{equation}\label{eq:theta-affine-new-rs}
\theta_1=a+qx^2,\qquad
\theta_2=\theta_3=a+q,\qquad
\theta_4=a+qy^2,\qquad
\theta_5=a+qz^2.
\end{equation}
By \eqref{B-spectrum-case(I)}, \(\theta_3=\tau=\sum_{k=1}^5\theta_k\). We plug in the expressions of $\theta_k$ from  \eqref{eq:theta-affine-new-rs} to obtain $
\theta_3
=
5a+q(x^2+2+y^2+z^2)$. 
By \eqref{eq:theta-affine-new-rs} again, $a=\theta_3-q$. We plug it into the above equation and solve 
\[
4\theta_3
=
q(3-x^2-y^2-z^2).
\]
Note that $y^2+z^2=(y+z)^2-2yz$. Combining this with \eqref{yz-identity} gives
\[
y^2+z^2 = (2-x)^2 - \frac{4(x-1)^2}{2-x} = \frac{(2-x)^3-4(x-1)^2}{2-x}.
\]
We plug $y^2+z^2$ into the above expression of $\theta_3$. It follows that 
\begin{equation}\label{eq:theta3-H-new-rs}
4\theta_3
=
q\,\frac{H(x)}{2-x},
\qquad
H(x):=2x^3-4x^2+x+2.
\end{equation}
Moreover, $\theta_1-\theta_3=q(x^2-1)$, as implied by \eqref{eq:theta-affine-new-rs}. It yields $\frac{\theta_1-\theta_3}{\theta_3}=\frac{4(2-x)(x^2-1)}{H(x)}$. 
Consequently, 
\begin{equation}\label{eq:R-new-rs}
\frac{\theta_1}{\theta_3}:=R(x)= 
1+\frac{4(2-x)(x^2-1)}{H(x)}.
\end{equation}

It remains to determine the admissible range of \(x\). Since \(y,z\) are real
and positive, we have a universal inequality: 
$(y+z)^2\geq4yz$. Combining this with \eqref{yz-identity} gives 
\[
(2-x)^2
\geq
\frac{8(x-1)^2}{2-x}.
\]
Since \(2-x>0\), it implies $
(2-x)^3-8(x-1)^2=
-x(x^2+2x-4)
\geq0$. In particular, $x^2+2x-4\leq 0$, which implies $x\leq \sqrt5-1$. 
Together with \(x>1\), this gives
\begin{equation}\label{eq:x-crude-new-rs}
1<x\leq\sqrt5-1.
\end{equation}
On this interval $(1, \sqrt{5}-1]$,
\[
R'(t)
= - \frac{8r_0(t)}{H(t)^2}, \qquad r_0(t):= 3t^3-6t^2+4t-2. 
\]
By direct calculations, $
r_0'(t)=(3t-2)^2>0$, and  $
r_0(\sqrt5-1)=-90+40\sqrt5<0$. 
Hence \(r_0(t)<0\), and therefore \(R'(t)>0\), for $t$  in the range specified by  
\eqref{eq:x-crude-new-rs}. In other words,
\[
\mbox{$R(\cdot)$ is strictly monotone increasing in $\bigl(1, \sqrt5-1\bigr]$}. 
\]
By \eqref{eq:ratio13<2}, $R(x)\leq 2$. However, a direct calculation gives $R(6/5)=18/7$, which is larger than $2$. 
Thus, the strict monotonicity of $R(\cdot)$ implies that 
\begin{equation}\label{eq:x-sharp-new-rs}
1<x< 6/5.
\end{equation}
In fact, we can obtain a more refined range of $x$ by solving $R(x)=2$. However, the upper bound $x<6/5$ is sufficient for our following steps.

\medskip
\noindent
{\bf Step 3: an upper bound for $\theta_5/\theta_3$}. 
By \eqref{eq:theta-affine-new-rs}, $\theta_3-\theta_5=q(1-z^2)$. Combining this with \eqref{eq:theta3-H-new-rs} gives 
\[
\frac{\theta_5}{\theta_3}= 1-\frac{\theta_3-\theta_5}{\theta_3}
=
1-\frac{4(2-x)(1-z^2)}{H(x)}. 
\]
Since \(y\geq z\) and
\(y+z=2-x\) (see \eqref{yz-identity}), we have $
z^2\leq\frac{(2-x)^2}{4}$. 
It follows that
\begin{equation}\label{eq:Phi-new-rs}
\frac{\theta_5}{\theta_3}\leq 
1-\frac{4(2-x)\left[1-(2-x)^2/4\right]}{H(x)}
=
\frac{x^3+2x^2-7x+2}
     {2x^3-4x^2+x+2}:=\Phi(x).
\end{equation}
We observe that 
\[
\Phi(x)+2
=
\frac{(x-1)(x+1)(5x-6)}{H(x)}.
\]
By \eqref{eq:theta3-H-new-rs}, \(H(x)>0\) (because $
\theta_3>0$ and $2-x>0$). 
Moreover, when $x$ satisfies \eqref{eq:x-sharp-new-rs}, 
$(x-1)(x+1)(5x-6)<0$, so that $\Phi(x)+2<0$. It follows that 
\begin{equation}\label{eq:theta5-lower-tail-rs}
\frac{\theta_5}{\theta_3}
\leq\Phi(x)<-2.
\end{equation}

\medskip
\noindent
{\bf Step 4: establishing the contradiction}. 
By \eqref{eq:theta5-lower-tail-rs} and the fact that $\theta_5<0<\theta_3$, we find that $|\theta_5|>2\theta_3$. Meanwhile, in \eqref{eq:ratio13<2}, we have seen that $\theta_1\leq 2\theta_3$. Together, they imply
\[
|\theta_5|>\theta_1.
\]
However, $\theta_1$ is the Perron eigenvalue of $B$. It determines the spectral radius of $B$, so that  $\theta_1=\max_i|\theta_i|$. This yields a contradiction.

\subsection{Case (II): nonrepeated positive eigenvalues under strict inequality}
\label{subsec:case-nonrep-strict}

In this case, $
\lambda_2>\lambda_3$, and  $4\lambda_1<9\lambda_3-S$. 
Let $B$ be the shifted matrix defined in Lemma~\ref{lem:shiftedB}, and
write $\tau=\tr(B)$. Lemma~\ref{lem:shiftedB} gives
\begin{equation}\label{B-spectrum-case(II)}
\theta_1>\theta_2>\theta_3=\tau>0>\theta_4\geq\theta_5, \qquad \tau>\frac{\theta_1}{2}. 
\end{equation}  
Lemma~\ref{lem:certificate} therefore yields a nonzero real symmetric
entrywise-nonnegative matrix $X$ such that $
BX=XB$ and $
B_{ij}X_{ij}=0$ for all $1\leq i,j\leq5$; moreover, 
the positive off-diagonal support graph of $X$ is $C_5$.
By Lemma~\ref{lem:cycle-rep}, there exist $a\in\mathbb R$ and $b>0$
such that $
B=aI_5+bX^2$. 
%and the positive off-diagonal support graphs of $B$ and $X$ are
%complementary five-cycles.

In contrast with Case~(I), we will no longer need the spectral identities of
Lemma~\ref{lem:C5poly}. The key additional fact in the present case is
that the high-trace inequality is strict: $\tau>\theta_1/2$. 

We now derive the contradiction. Since $B$ and $X$ are real symmetric and commute, they are
simultaneously orthogonally diagonalizable. Hence there exists a
common orthogonal matrix $Q$ such that 
\[
B=Q\diag(\theta_1,\theta_2,\theta_3,\theta_4,\theta_5)Q\T,
\qquad
X=Q\diag(x_1,x_2,x_3,x_4,x_5)Q\T.
\]
Here $x_i$ denotes the eigenvalue of $X$ paired with $\theta_i$.
Since $B_{ii}>0$ and $B_{ii}X_{ii}=0$, the matrix $X$ has zero
diagonal. Thus
\[
\tr(X)=\sum_{i=1}^5x_i=0.
\]
For $\eps>0$, define
\begin{equation}\label{eq:P-eps-case-II}
\alpha_\eps=\frac{\eps x_3}{4},
\qquad
P_\eps=B+\eps X+\alpha_\eps I_5.
\end{equation}
We show that, for all sufficiently small $\eps>0$, the matrix  $P_\eps$ satisfies 
\begin{equation}\label{case(II)-goal}
\mbox{$P_\eps$
is strictly positive}, \qquad \lambda_3(P_\eps)=\tr(P_\eps),
\qquad
\tr(P_\eps)>\frac{\lambda_1(P_\eps)}{2}.
\end{equation}
This will contradict Lemma~\ref{lem:positive-strict}.

We now show \eqref{case(II)-goal}. Consider the first claim. By
Lemma~\ref{lem:cycle-rep}, the off-diagonal support graphs of $B$ and
$X$ are complementary five-cycles. Hence, for every $i\neq j$,
exactly one of $B_{ij}$ and $X_{ij}$ is positive. Since $\eps>0$,
every off-diagonal entry of $
B+\eps X$ 
is therefore strictly positive. Moreover, $B_{ii}>0$ for every $i$,
while $\alpha_\eps\to0$ as $\eps\to0$. Thus, for all sufficiently
small $\eps>0$,
\[
(P_\eps)_{ii}=B_{ii}+\alpha_\eps>0.
\]
Consequently, $P_\eps$ is entrywise strictly positive.

Consider the second claim. The five eigenvalues of $P_\eps$ are
\[
\theta_i+\eps x_i+\alpha_\eps,
\qquad i=1,\ldots,5.
\]
By \eqref{B-spectrum-case(II)}, $
\theta_1>\theta_2>\theta_3>\theta_4\geq\theta_5$. 
Hence, for all sufficiently small $\eps>0$, the first three perturbed
eigenvalues remain the three largest, in the same order. In particular,
$\lambda_3(P_\eps)
=
\theta_3+\eps x_3+\alpha_\eps$. 
Using $\theta_3=\tau$ and
$\alpha_\eps=\eps x_3/4$, we obtain
\[
\lambda_3(P_\eps)
=
\tau+\frac{5\eps x_3}{4}.
\]
On the other hand, since $\tr(X)=0$,
\[
\tr(P_\eps)
=
\tr(B)+\eps\tr(X)+5\alpha_\eps
=
\tau+\frac{5\eps x_3}{4}.
\]
Comparing the above two equations gives $\lambda_3(P_\eps)=\tr(P_\eps)$. 

Consider the third claim. 
For sufficiently small $\eps$, the largest eigenvalue of $P_\eps$ is $
\lambda_1(P_\eps)
=
\theta_1+\eps x_1+\alpha_\eps$. 
Hence, it follows from \eqref{eq:P-eps-case-II} that 
\begin{align} \label{matrix-Peps}
2\tr(P_\eps)-\lambda_1(P_\eps)
&=
2(\tau+5\alpha_{\epsilon}) -(\theta_1
+\eps x_1+\alpha_\eps)\cr
&=
2\tau-\theta_1+9\alpha_\eps -\epsilon x_1 \cr
&= 
2\tau-\theta_1 +\eps\left(\frac{9x_3}{4}-x_1\right).
\end{align}
By \eqref{B-spectrum-case(II)}, the constant term
$2\tau-\theta_1$ is strictly positive. We do not know the sign of  $9x_3/4-x_1$, but as long as $\eps$ is sufficiently small, it can be guaranteed that $
2\tr(P_\eps)-\lambda_1(P_\eps)>0$. 
This completes the proof of \eqref{case(II)-goal} and establishes the contradiction.

\subsection{Case (III): nonrepeated positive eigenvalues on the equality boundary}
\label{subsec:case-nonrep-boundary}

In this case, $
\lambda_2>\lambda_3$, and $
4\lambda_1=9\lambda_3-S$. 
Let $B$ be the shifted matrix defined in Lemma~\ref{lem:shiftedB}, and write
$\tau=\tr(B)$. Lemma~\ref{lem:shiftedB} gives
\begin{equation}\label{B-spectrum-case(III)}
\theta_1>\theta_2>\theta_3=\tau>0>\theta_4\geq\theta_5, \qquad \tau = \frac{\theta_1}{2}.  
\end{equation}
Lemma~\ref{lem:certificate} yields a nonzero real symmetric
entrywise-nonnegative matrix $X$ such that $
BX=XB$, $
B_{ij}X_{ij}=0$ for all $1\leq i,j\leq5$, and the positive off-diagonal support graph of $X$ is $C_5$. By
Lemma~\ref{lem:cycle-rep}, there exist $a\in\mathbb R$ and $b>0$ such that $
B=aI_5+bX^2$, 
and the positive off-diagonal support graphs of $B$ and $X$ are complementary
five-cycles. Lemma~\ref{lem:C5poly} also applies to $X$.

The difference from Case~(II) is that the strict high-trace margin has
disappeared: here $2\tau-\theta_1=0$. Thus the perturbation argument from
Case~(II) cannot by itself give a contradiction, and we must use the spectral
identities of the weighted five-cycle.

\medskip
\noindent
{\bf Step 1: the spectrum of $X$.}
Since $B$ and $X$ are real symmetric matrices which commute, they are simultaneously
orthogonally diagonalizable. Therefore, there exists a common orthonormal eigenbasis such that the eigenvalues of $X$ are paired with the eigenvalues of $B$. We use
$x_i$ denote the eigenvalue of $X$ paired with $\theta_i$. From the identity $
B=aI_5+bX^2$, we have 
\begin{equation}\label{eq:paired-caseIII}
\theta_i=a+bx_i^2,
\qquad i=1,\ldots,5.
\end{equation}
Let $r>0$ be the Perron eigenvalue of $X$. Since the support graph of $X$ is
$C_5$, the matrix $X$ is irreducible and nonnegative. Hence $r$ is simple and
$|x_i|\leq r$ for every eigenvalue $x_i$ of $X$. Since $b>0$, the eigenvalue
$a+br^2$ is also the largest eigenvalue of $B$. Thus we have $x_1=r$. 
Combining \eqref{eq:paired-caseIII} with \eqref{B-spectrum-case(III)} gives
\begin{equation}\label{eq:abs-order-bn}
r=x_1>|x_2|>|x_3|>|x_4|\geq|x_5|.
\end{equation}

By Lemma~\ref{lem:C5poly}, $X$ has three positive eigenvalues and two negative eigenvalues. We have seen that $x_1$ is a positive eigenvalue. It remains to determine the signs of other eigenvalues. 

First, we determine the sign of $x_3$. Denote $
s=|x_3|$ and $Q=\sum_{i=1}^5x_i^2$ for brevity. Then, 
\[
\theta_1=a+br^2, \qquad \theta_3=a + bs^2, \qquad  \tr(B)=5a + b\sum_{i=1}^5x_i^2=5a + bQ. 
\]
By \eqref{B-spectrum-case(III)}, $\theta_3=\tr(B)$ and $\theta_1=2\theta_3$. These two equalities imply 
\[
4a+b(Q-s^2)=0, \qquad a+br^2=2(a+bs^2).
\] 
The second equality gives $a=b(r^2-2s^2)$. We plug it into the first equation to solve 
\begin{equation}\label{eq:9s-bn}
9s^2=Q+4r^2.
\end{equation}
Recall that $s=|x_3|$. By \eqref{eq:abs-order-bn}, $|x_2|>s$. It follows that $Q\geq |x_1|^2+2|x_3|^2>r^2+2s^2$. 
Therefore, \eqref{eq:9s-bn} implies $9s^2>5r^2+2s^2$, and hence  
\begin{equation}\label{eq:sbound-bn}
s> \sqrt{\frac57}  \,r>\frac{4r}{9}.
\end{equation}
Now, suppose the sign of $x_3$ is positive. For $\eps>0$, set $
\alpha_\eps=\frac{\eps x_3}{4}$ and $
P_\eps=B+\eps X+\alpha_\eps I_5$, which is the same as in \eqref{eq:P-eps-case-II}. 
As in Case~(II), for all sufficiently small $\eps>0$, the matrix $P_\eps$ is
entrywise strictly positive and the first three perturbed eigenvalues remain
in the same order. Using similar calculations as in \eqref{matrix-Peps}, we have 
\[
2\tr(P_\eps)-\lambda_1(P_\eps)
=
\eps\left(\frac{9x_3}{4}-x_1\right)= \eps\left(\frac{9x_3}{4}-r\right)>0.
\]
However, Lemma~\ref{lem:positive-strict} gives
$\lambda_3(P_\eps)<\tr(P_\eps)$, which is a contradiction. Therefore, $x_3>0$ is impossible. We conclude that 
\begin{equation} \label{caseIII-sign-x3}
x_3<0.
\end{equation}

Next, we determine the sign of $x_2$.  Suppose for contradiction that $x_2>0$.  By \eqref{eq:abs-order-bn}, 
 $x_2>|x_3|=s$. By Lemma~\ref{lem:C5poly}, $X$ has exactly
three positive and two negative eigenvalues. Thus we may write its eigenvalues
as
\[
x_1=r,\quad  x_2,\quad x_4=p,\quad x_3=-s, \quad x_5=-u,
\qquad \mbox{for some }p,u>0.
\]
Since the support of $X$ is $C_5$, this matrix is irreducible and nonnegative.
Therefore every eigenvalue other than the Perron eigenvalue $r$ is strictly smaller than
$r$ in absolute value. This implies $u<r$. Since $X$ has zero diagonal, $
\tr(X)=0$. Therefore, $r+x_2+p=s+u$. 
But the left-hand side is strictly larger than $r+s$ (because $x_2>s$), whereas the right-hand
side is strictly smaller than $r+s$. This yields a contradiction. We thus conclude that 
\begin{equation} \label{caseIII-sign-x2}
x_2<0
\end{equation}

Since $x_2$ and $x_3$ are negative and $x_1$ is positive, we immediately know that $x_4$ and $x_5$ are positive (because $X$ has exactly three positive eigenvalues and two negative eigenvalues). 
Given \eqref{eq:abs-order-bn} and the signs of eigenvalues and noting that $\theta_2>\theta_3$, we re-parametrize them by  
\begin{equation}\label{eq:Xroots-bn}
(x_1, x_2, x_3, x_4, x_5)=(r, -u, -v, p, q_0),\qquad
\mbox{where}\quad r>u>v>p\geq q_0>0.
\end{equation}

\medskip
\noindent
{\bf Step 2: spectral identities for $X$.}
We collect a few useful identities for the values $u,v,p,r,q_0$. 
From \eqref{eq:paired-caseIII} and $\theta_1=2\theta_3$ (see \eqref{B-spectrum-case(III)}), we obtain $
a+br^2=2(a+bv^2)$. It leads to 
\begin{equation}\label{eq:a-v-bn}
a=b(r^2-2v^2).
\end{equation}
Also, since $\tr(B)=\theta_3$, we have $\theta_1+\theta_2+\theta_4+\theta_5=0$. Combining it with \eqref{eq:paired-caseIII} gives $4a + b(r^2+u^2+p^2+q_0^2)=0$. 
Substituting the expression of $a$ in \eqref{eq:a-v-bn} gives the following  identity:
\begin{equation}\label{eq:balance-bn}
8v^2=5r^2+u^2+p^2+q_0^2.
\end{equation}
We will refer to \eqref{eq:balance-bn} as the ``balance identity". This is a special property unique to Case (III), which lies on the equality boundary. 

We then apply the spectral identities in 
Lemma~\ref{lem:C5poly}. Let $X$ be parameterized as in Lemma~\ref{lem:C5poly}, with nonzero entries $y_1, \ldots, y_5$. Meanwhile, $X$ has eigenvalues
$r,-u,-v,p,q_0$.    Applying the first four equalities in part (d)  of Lemma~\ref{lem:C5poly}, it follows that 
\begin{align}
& r-u-v+p+q_0=0, \qquad 
r^3-u^3-v^3+p^3+q_0^3=0,\cr
& r^2+u^2+v^2+p^2+q_0^2=2\sum_{i=1}^5y_i^2,
\qquad ruvpq_0=2\prod_{i=1}^5y_i.\label{eq:X-spectral-identities}
\end{align}

\medskip
\noindent
{\bf Step 3: a lower bound for $\prod_{i=1}^5 y_i$}. 
We aim to derive a lower bound for $\prod_{i=1}^5 y_i$ in terms of $(p+q_0)^2$. 
Since $B$ has strictly positive diagonal entries and $B=aI_5+bX^2$, we have 
\[
B_{ii}=a+b(y_{i-1}^2+y_i^2)>0,
\qquad \mbox{for all } i\in\mathbb Z/5\mathbb Z.
\]
Here, the indices should be read modulo 5. 
Plugging in $a=b(r^2-2v^2)$ as in 
\eqref{eq:a-v-bn} and using $b>0$, we obtain
\[
y_{i-1}^2+y_i^2>2v^2-r^2,
\qquad \mbox{for all }i\in\mathbb Z/5\mathbb Z.
\]
Fix an index $i$. Since $y_j=y_{j+5}$ (i.e., the indices are read modulo 5), we have
\begin{equation} \label{eq:neighbor-direct-bn}
3(2v^2-r^2) <  \sum_{k\in \{0, 1, 3\}}(y_{i+k-1}^2+y_{i+k}^2) = y_i^2+\sum_{\ell=-1}^3 y_{i+\ell}^2 = y_i^2+\sum_{j=1}^5y_j^2. 
\end{equation}
By \eqref{eq:X-spectral-identities}, $\sum_{j=1}^5y_j^2=\frac{1}{2}(r^2+u^2+v^2+p^2+q_0^2)$.
We apply the balance identity in \eqref{eq:balance-bn} to eliminate $u^2, p^2, q_0^2$. It gives 
\begin{equation}\label{eq:sum-y-direct-bn}
\sum_{j=1}^5y_j^2=\frac{9v^2-4r^2}{2}.
\end{equation}
Combining \eqref{eq:neighbor-direct-bn} and \eqref{eq:sum-y-direct-bn} gives 
\begin{equation}\label{eq:yi-first-bn}
y_i^2>\frac{3v^2-2r^2}{2},
\qquad \mbox{for each of } i=1,\ldots,5.
\end{equation}
We apply the balance identity in \eqref{eq:balance-bn} again. It gives $8v^2
=5r^2+u^2+p^2+q_0^2$. By \eqref{eq:abs-order-bn}, $u>v$. We then have 
 $7v^2>5r^2+p^2+q_0^2$, which gives $
v^2>\frac{5r^2+p^2+q_0^2}{7}$. 
Substituting this into \eqref{eq:yi-first-bn} yields
\begin{equation}\label{eq:yi-refined-bn}
y_i^2>
\frac{r^2+3p^2+3q_0^2}{14},
\qquad i=1,\ldots,5.
\end{equation}
This gives a lower bound for $y_i^2$ in terms $r^2$, $p^2$, and $q_0^2$, but our goal is to obtain a lower bound in terms of $(p+q_0)^2$. For this purpose, we need the following inequalities:
\begin{equation}\label{eq:two-spectral-estimates-bn}
r^2>\frac76(p+q_0)^2,
\qquad
pq_0<\frac{16}{343}(p+q_0)^2.
\end{equation}
Suppose \eqref{eq:two-spectral-estimates-bn} is true. Then, 
\[
p^2+q_0^2
=
(p+q_0)^2-2pq_0
>
\frac{311}{343}(p+q_0)^2. 
\]
We plug this result and the lower bound on $r^2$ in \eqref{eq:two-spectral-estimates-bn} into  \eqref{eq:yi-refined-bn}.  It follows that 
\begin{align*}
y_i^2
&>
\frac1{14}
\left(
\frac76+3\frac{311}{343}
\right)
(p+q_0)^2 =
\frac{7999}{28812}(p+q_0)^2
>
\frac{13}{47}(p+q_0)^2.
\end{align*}
Thus
\begin{equation}\label{eq:prod-lower-new-bn}
\prod_{i=1}^5y_i^2
>
(p+q_0)^{10}
\left(\frac{13}{47}\right)^5.
\end{equation}
This provides the desirable lower bound. 

It remains to show \eqref{eq:two-spectral-estimates-bn}. Write for brevity 
\[
m=p+q_0,  \qquad T=\frac{r}{m},
\qquad
\eta=\frac{u-v}{m},
\qquad
\gamma=\frac{pq_0}{m^2}.
\]
Then, \eqref{eq:two-spectral-estimates-bn} translate to 
\begin{equation}\label{eq:two-spectral-estimates-bn-new}
T^2 >\frac76,
\qquad
\gamma <\frac{16}{343}. 
\end{equation}
From now on, we aim to lower bound $T$ and upper bound $\gamma$. 

To show \eqref{eq:two-spectral-estimates-bn-new}, the key is to obtain some equalities about $T$ and $\gamma$. 
First, we use the spectral identities in \eqref{eq:X-spectral-identities}. It implies $r-u-v+p+q_0=0$. Therefore,  
\begin{equation}\label{eq:uvsum-direct-bn}
u+v=r+m.
\end{equation}
By \eqref{eq:X-spectral-identities} again, $r^3-u^3-v^3+p^3+q_0^3=0$. Note that for real values $z_1, z_2$, $z_1^3+z_2^3=(z_1+z_2)^3-3z_1z_2(z_1+z_2)$. As a result, 
\begin{align*}
(u^3+v^3)-(p^3+q_0^3)  & = [(u+v)^3-3uv(u+v)]- (m^3-3pq_0m) \cr
&= [(r+m)^3- 3uv(m+r)]- (m^3 -3 pq_0m)\cr
&=  r^3 + 3rm(m+r)-3uv(m+r)+3pq_0m. 
\end{align*}
Plugging them into $r^3-u^3-v^3+p^3+q_0^3=0$ gives 
\begin{equation}\label{eq:uvprod-direct-bn}
uv=rm+\frac{pq_0m}{r+m}.
\end{equation}
Note that $(u-v)^2=(u+v)^2-4uv$. Combining this with \eqref{eq:uvsum-direct-bn} and \eqref{eq:uvprod-direct-bn}, we have 
\[
(u-v)^2=(r+m)^2-4\Bigl(rm+\frac{pq_0m}{r+m}\Bigr)=(r-m)^2-\frac{4pq_0m}{r+m}.
\]
Dividing the above equation 
by $m^2$ on  both sides  gives 
\begin{equation}\label{eq:eta-direct-bn}
\eta^2=(T-1)^2-\frac{4\gamma}{T+1}.
\end{equation}
Next, we invoke the balance identity in \eqref{eq:balance-bn}, which can be re-written as $8\frac{v^2}{m^2}=5T^2+\frac{u^2}{m^2}+\frac{p^2+q_0^2}{m^2}$. 
We substitute
\[
\frac{u}{m}=\frac{(r+m)+(u-v)}{2m}=\frac{T+1+\eta}{2},
\qquad
\frac{v}{m}=\frac{T+1-\eta}{2},
\qquad
\frac{p^2+q_0^2}{m^2}=1-2\gamma
\]
into the balance identity, and then use \eqref{eq:eta-direct-bn} to eliminate $(T-1)^2$. It gives
\begin{equation}\label{eq:balance-normalized-direct-bn}
18(T+1)\eta
=
-6T^2+10+
\left(8-\frac{28}{T+1}\right)\gamma.
\end{equation}

We now use \eqref{eq:eta-direct-bn}-\eqref{eq:balance-normalized-direct-bn} to show \eqref{eq:two-spectral-estimates-bn-new}. By definition, $\gamma\leq \frac{(p+q_0)^2}{4m^2}=\frac{1}{4}$. Moreover, since $r>u>v$, \eqref{eq:uvsum-direct-bn}  gives $u-v<r-v=u-m<r-m$. It follows that 
\begin{equation}\label{eq:ratio-diff-direct-bn}
0<\eta<T-1, \qquad 0 < \gamma\leq 1/4. 
\end{equation}
If $T\geq5/2$, then $0< 8-\frac{28}{T+1}\leq 8$. The last term on the right-hand side of
\eqref{eq:balance-normalized-direct-bn} is at most $8\gamma\leq 2$. Hence,  the right-hand side
is at most $-6T^2+12<0$, whereas the left-hand side is positive. This is impossible. Therefore, $
T<5/2$. 
In this range, $
8-\frac{28}{T+1}<0$, and the right hand side of \eqref{eq:balance-normalized-direct-bn} is at most $10-6T^2$. Since the left hand side is positive, we immediately have $T^2<5/3$. Thus, 
\begin{equation}\label{eq:T-crude-direct-bn}
1<T<\sqrt{\frac53}.
\end{equation}
Note that \eqref{eq:eta-direct-bn} gives $\eta^2=(T-1)^2-\frac{4\gamma}{T+1}$ and 
\eqref{eq:balance-normalized-direct-bn} gives $18(T+1)\eta=-6T^2+10-(5-2T)\frac{4\gamma}{T+1}$. 
We multiply the first equality by $(5-2T)$ and subtract it from the second equality. This eliminates $\gamma$ and yields
\begin{equation}\label{eq:C-direct-bn}
18(T+1)\eta-(5-2T)\eta^2= 2T^3-15T^2+12T+5.
\end{equation}
Let $C(T)=18(T+1)\eta-(5-2T)\eta^2$. 
For each $t$ in the range of \eqref{eq:T-crude-direct-bn}, define
\[
f_t(\xi)=18(t+1)\xi-(5-2t)\xi^2.
\]
For $0<\xi<t-1$, $
f_t'(\xi)
=
18(t+1)-2(5-2t)\xi > 18(t+1)-2(5-2t)(t-1)
=
4t^2+4t+28>0$. 
Hence $f_t(\xi)$ is strictly increasing for $\xi\in (0,t-1)$. Since  $0<\eta<T-1$ (see \eqref{eq:ratio-diff-direct-bn}), we have 
\begin{equation} \label{fT(T-1)}
C(T)=f_T(\eta) <f_T(T-1).
\end{equation}
Note that $f_T(T-1)=18(T+1)(T-1)-(5-2T)(T-1)^2$. Meanwhile, \eqref{eq:C-direct-bn} implies $C(T)=2T^3-15T^2+12T+5$. 
A direct calculation gives
\[
f_T(T-1)-C(T)=24T^2-28. 
\]
Therefore, \eqref{fT(T-1)} implies $T^2>\frac76$. 
This proves the first inequality in \eqref{eq:two-spectral-estimates-bn-new}.

Moreover, $C(T)=f_T(\eta)>f_T(0)=0$. On the interval
\eqref{eq:T-crude-direct-bn}, $
C'(T)=6T^2-30T+12<0$, indicating that $C(T)$ is strictly monotone decreasing on this interval. Meanwhile, 
$
C(9/7)=-\frac{40}{343}<0$. It implies that  $T<9/7$. Together with $T^2>7/6$, we have 
\begin{equation}\label{case(III)-T-range}
\sqrt{7/6}<T<9/7. 
\end{equation}
Using \eqref{eq:eta-direct-bn} and $\eta>0$, we can deduce
\begin{equation} \label{case(III)-gamma-range}
\gamma<\frac{(T+1)(T-1)^2}{4}<\frac{(9/7+1)(9/7-1)^2}{4}=\frac{16}{343}.
\end{equation}
This proves the second inequality in \eqref{eq:two-spectral-estimates-bn-new}. Step~3 is now complete.

\medskip
\noindent
{\bf Step 4: an upper bound for $\prod_{i=1}^5 y_i$ and the contradiction.}
By one of the spectral identities in \eqref{eq:X-spectral-identities}, $ruvpq_0=2\prod_{i=1}^5y_i$. Recall that $m=p+q_0$, $T=\frac{r}{m}$, and $\gamma=\frac{pq_0}{m^2}$. We immediately have
\begin{equation}\label{case(III)-final}
\prod_{i=1}^5y_i^2 =
\left(\frac{ruvpq_0}{2}\right)^2 = \frac{m^{10}}{2^2}T^2\gamma^2 \Bigl(\frac{uv}{m^2}\Bigr)^2. 
\end{equation}
By \eqref{eq:uvsum-direct-bn} and $u\neq v$ (see \eqref{eq:Xroots-bn}), we have $uv\leq (u+v)^2/4 \leq \frac{(r+m)^2}{4}\leq \frac{m^2}{4}(T+1)^2$. By \eqref{case(III)-T-range}, $T<9/7$. Combining these arguments gives 
\[
\frac{uv}{m^2}<\frac{64}{49}. 
\]
We plug this upper bound into \eqref{case(III)-final} and also apply the upper bounds of $T$ and $\gamma$ in \eqref{case(III)-T-range}-\eqref{case(III)-gamma-range}. It follows that 
\begin{align}
\prod_{i=1}^5y_i^2< 
(p+q_0)^{10}
\left(\frac{4608}{117649}\right)^2 <
(p+q_0)^{10}
\left(\frac{2}{51}\right)^2.
\label{eq:prod-upper-new-bn}
\end{align}
However, this contradicts the lower bound for $\prod_{i=1}^5y_i^2$ in \eqref{eq:prod-lower-new-bn}, because 
$\left(\frac{13}{47}\right)^5
>
\left(\frac{2}{51}\right)^2$.

\begin{proof}[Conclusion of the proof of Theorem~\ref{thm:main}]
Every $\lambda\in \Wreg$ falls into one of Cases~(I)--(III). The analyses in Sections~\ref{subsec:case-repeated}-\ref{subsec:case-nonrep-boundary} exclude all possibilities. Hence, none of $\lambda\in \Wreg$ is
symmetrically nonnegatively realizable.
\end{proof}

\section{Discussion}

The symmetric nonnegative inverse eigenvalue problem is a difficult but
interesting problem, with applications such as network modeling. In this
paper, we study the residual order-five problem and identify the explicit
nonempty region $\Wreg=\{\lambda\in\R:4\lambda_1\le 9\lambda_3-S\}$, 
in which no symmetric nonnegative realization exists. To the best of our
knowledge, this impossibility region has not been identified previously.

The proof combines several ideas. The main device is the diagonal shift
$B=A+\frac{\lambda_3-S}{4}I_5$,  
which moves the original low-trace problem to the critical high-trace level
$\lambda_3(B)=\operatorname{trace}(B)$.  
At this critical level, we construct a nonzero nonnegative matrix $X$ that
commutes with $B$ and is supported on zero entries of $B$. Using this
construction, we show that the support graph of $X$ is a five-cycle and,
in the relevant cases, that
$B=aI_5+bX^2$ 
for suitable scalars $a$ and $b$. These relations allow us to connect the
spectra of $B$ and $X$ and ultimately derive the desired contradiction.

Several questions remain open. First, the complementary region
$\R\setminus\Wreg$ remains unclear, and it would be interesting to obtain a
finer description of this region, either through additional impossibility
conditions or through new realization constructions. Second, the
$6\times6$ SNIEP is a natural next case: one may ask whether analogous
residual regions can be identified and whether the shift, commutation, and
support-reduction ideas developed here remain useful in higher dimensions.
Finally, the connection in Section~\ref{sec:network} motivates further study
of the symmetric doubly stochastic inverse eigenvalue problem (SDIEP),
including the special spectral families arising from network modeling.

% REVISION: acknowledgements made unnumbered.
\section*{Acknowledgements} 
The authors used OpenAI’s GPT-5.6 in preparing this manuscript. All mathematical arguments were independently verified by the authors.

\section{Appendix}\label{sec:appendix}

This section contains the proofs of Lemmas~\ref{lem:shiftedB}--\ref{lem:C5poly}.  The only external theorem used in these proofs is the following high-trace result of \citet{LoewySpector2017}.  We record it here rather than in the main sections.

\begin{theorem}[Loewy--Spector high-trace condition]\label{thm:LS-hightrace} % REVISION: added stable cross-reference label.
Suppose $C$ is a real symmetric entrywise-nonnegative $5\times5$ matrix with ordered eigenvalues $
\eta_1\ge\eta_2\ge\eta_3\ge\eta_4\ge\eta_5$ and trace $\tau_C$.  If $\tau_C\ge\frac{\eta_1}{2}$, then $
\eta_3\le\tau_C$.
\end{theorem}

\subsection{Proof of Lemma~\ref{lem:shiftedB}}
Since $S<\lambda_3$, we have $c>0$.
We first prove that $A$ is irreducible.  Suppose otherwise.  Since $A$ is symmetric, after a simultaneous permutation of rows and columns we may write $A$ as a direct sum of irreducible entrywise-nonnegative blocks.  Every such block has nonnegative trace, and its Perron root $\lambda_1$ is a nonnegative eigenvalue of that block whose modulus dominates that of every other eigenvalue in the block.

If $\lambda_4$ and $\lambda_5$ lie in distinct blocks, each of those two blocks must contain a nonnegative Perron eigenvalue at least as large in modulus as the negative eigenvalue in that block.  These two Perron eigenvalues are two distinct members of $\{\lambda_1,\lambda_2,\lambda_3\}$.  Pairing each negative eigenvalue with the Perron root of its block gives a nonnegative contribution to the trace.  The remaining positive eigenvalue is at least $\lambda_3$, and every remaining block has nonnegative trace.  Hence $S\ge\lambda_3$, contradicting $S<\lambda_3$.

If $\lambda_4$ and $\lambda_5$ lie in the same block and some positive eigenvalue lies outside that block, then the block containing the two negative eigenvalues has nonnegative trace, while the blocks outside it contain only positive eigenvalues and therefore contribute at least $\lambda_3$ in total.  Again $S\ge\lambda_3$, a contradiction.  Thus all three positive eigenvalues and both negative eigenvalues lie in a single block.  Hence $A$ is irreducible.

We now prove the claims. Recall that $B=A+cI_5$, where $c>0$. Therefore, the irreducibility of $A$ implies that $B$ is irreducible as well. The entry-wise non-negativity of $A$ implies that $B$ has strictly positive diagonal entries. The eigenvalues of $B$ are $\theta_i=\lambda_i+c$. Therefore, $
\theta_1>\theta_2\geq\theta_3\geq \theta_4\geq\theta_5$. 
Write $\tau=\tr(B)$. It remains to show:
\begin{equation} \label{lem-shiftB-0}
\tau=\theta_3, \qquad \tau\geq \theta_1/2, \qquad \theta_4<0. 
\end{equation}
Recall that $\theta_i=\lambda_i+c$, $S=\tr(A)=\sum_{k=1}^5\lambda_k$, and $c=\frac{\lambda_3-S}{4}$. By direct calculations,  
\[
\tau =\sum_{k=1}^5\theta_k =S+5c
=\frac{5\lambda_3-S}{4}
=\lambda_3+c
=\theta_3
\]
This proves the first claim in \eqref{lem-shiftB-0}. Moreover, 
\[
2\tau-\theta_1 = 2(S+5c)-(\lambda_1+c)  = 2S-\lambda_1+9c =  \frac{9\lambda_3-S-4\lambda_1}{4}\ge0.
\]
This proves the second claim in \eqref{lem-shiftB-0}. Finally, Perron dominance for $A$ gives $\lambda_1+\lambda_5\ge0$. We have also assumed $\lambda_2>0>\lambda_4$. It follows that 
\[
\theta_4 = \lambda_4+\frac{\lambda_3-S}{4} =\frac{3\lambda_4-\lambda_2-(\lambda_1+\lambda_5)}{4} \le \frac{3\lambda_4-\lambda_2}{4}<0. 
\]
This proves the third claim in \eqref{lem-shiftB-0}.\qed

\subsection{Proof of Lemma~\ref{lem:positive-strict}}\label{app:positive-strict}

\begin{proof}[Proof of Lemma~\ref{lem:positive-strict}]
Theorem~\ref{thm:LS-hightrace}  gives $\mu_3\le\tau_P$.  Suppose for contradiction that
$\mu_3=\tau_P$.

Since $P$ is strictly positive, it is primitive.  Perron--Frobenius therefore gives
\[
\mu_1>\mu_2.
\]
Choose an orthogonal matrix $Q$ such that
\[
P=Q\diag(\mu_1,\mu_2,\mu_3,\mu_4,\mu_5)Q\T.
\]
For $\eps>0$, set
\[
D_{\eps}=\diag(\mu_1,\mu_2+\eps,\mu_3+\eps,\mu_4-\eps,\mu_5-\eps),
\qquad
P_{\eps}=QD_{\eps}Q\T.
\]
Choose $\eps$ small enough that
\[
\mu_2+\eps<\mu_1.
\]
Then the entries of $D_{\eps}$ remain in nonincreasing order.  Also, by continuity, $P_{\eps}$ remains entrywise strictly positive for all sufficiently small $\eps>0$.

The perturbation preserves the trace:
\[
\tr(P_{\eps})=\tau_P.
\]
Its Perron root remains $\mu_1$, whereas its third eigenvalue is
\[
\lambda_3(P_{\eps})=\mu_3+\eps=\tau_P+\eps.
\]
Moreover $\tau_P\ge\mu_1/2$ is unchanged.  Applying  Theorem~\ref{thm:LS-hightrace}  to $P_{\eps}$ yields
\[
\tau_P+\eps=\lambda_3(P_{\eps})\le\tr(P_{\eps})=\tau_P,
\]
a contradiction.  Hence $\mu_3<\tau_P$.
\end{proof}

\subsection{Proof of Lemma~\ref{lem:certificate}}\label{app:certificate}

This subsection proves the complementary commuting matrix and all of its support conclusions.  Parts 1 and 2 establish the general statement that the support is either $C_5$ or $K_2\cup P_3$.  Part 3 treats the case $\theta_2>\theta_3$, and Part 4 treats the case $\theta_2=\theta_3=\tau$. 

\begin{proof}[Proof of Lemma~\ref{lem:certificate}]
The proof has four parts.  In Part 1, we construct a nonzero matrix $X$ satisfying \eqref{eq:certificate}.  The key point is first to show that an orthogonal perturbation cannot make all zero off-diagonal entries of $B$ increase at once; a separation theorem then produces $X$.  In Part 2, we classify the possible support graphs of $X$ and show that only $C_5$ or $K_2\cup P_3$ can remain.   In Part 3, we exclude $K_2\cup P_3$ when $\theta_2>\theta_3$.  In Part 4, we exclude $K_2\cup P_3$ when $\theta_2=\theta_3=\tau$.   We now consider Part 1.

\medskip
\noindent\textbf{Part 1: construction of $X$.}
Let
\[
E=\{\{i,j\}:1\le i<j\le5,\ B_{ij}=0\}
\]
be the set of off-diagonal zero positions of $B$.  If $E$ were empty, then $B$ would be entrywise strictly positive, and Lemma~\ref{lem:positive-strict} would give $\theta_3<\tr(B)$, contradicting $\theta_3=\tr(B)$.  Hence $E$ is nonempty.

Write $m=|E|$ and list the elements of $E$ as
\[
E=\{e_1,\ldots,e_m\},\qquad e_\ell=\{i_\ell,j_\ell\}.
\]
Let $\mathcal K$ be the real vector space of skew-symmetric $5\times5$ matrices.  Define
\[
L:\mathcal K\longrightarrow\mathbb R^m,
\qquad
L(K)_\ell=(BK-KB)_{i_\ell j_\ell},\quad 1\le\ell\le m.
\]
Thus $L(K)$ records the first-order changes, under an orthogonal perturbation, of precisely those off-diagonal entries of $B$ that are zero.  Let
\[
\mathcal O=(0,\infty)^m
\]
be the open positive orthant in $\mathbb R^m$.

We first prove
\begin{equation}\label{eq:disjoint}
L(\mathcal K)\cap\mathcal O=\varnothing.
\end{equation}
Suppose otherwise.  Then there is a skew-symmetric matrix $K\in\mathcal K$ such that
\[
(BK-KB)_{i_\ell j_\ell}>0,\qquad \ell=1,\ldots,m.
\]
Set
\[
Q_t=e^{tK},\qquad B_t=Q_t\T BQ_t=e^{-tK}Be^{tK}.
\]
Because $K\T=-K$,
\[
Q_t\T Q_t=e^{tK\T}e^{tK}=e^{-tK}e^{tK}=I,
\]
so $Q_t$ is orthogonal.  Differentiating $B_t=e^{-tK}Be^{tK}$ at $t=0$ gives
\[
\left.\frac{d}{dt}B_t\right|_{t=0}=BK-KB.
\]
Therefore, at each off-diagonal position where $B$ is zero,
\[
(B_t)_{i_\ell j_\ell}
=t(BK-KB)_{i_\ell j_\ell}+o(t)>0
\]
for all sufficiently small $t>0$.  Every other off-diagonal entry of $B$ is already positive, and every diagonal entry is strictly positive; by continuity, these entries remain positive for sufficiently small $t$.  Hence $B_t$ is entrywise strictly positive.

Since $B_t$ is orthogonally similar to $B$, it has the same eigenvalues and the same trace.  Thus
\[
\lambda_3(B_t)=\theta_3=\tau=\tr(B_t),
\qquad
\tau\ge\frac{\theta_1}{2}=\frac{\lambda_1(B_t)}2.
\]
This contradicts Lemma~\ref{lem:positive-strict}.  Hence \eqref{eq:disjoint} holds.

We now use \eqref{eq:disjoint} to construct $X$.  The set $L(\mathcal K)$ is a linear subspace of $\mathbb R^m$, while $\mathcal O$ is a nonempty open convex set disjoint from it.  By the finite-dimensional separation theorem, there exists a nonzero vector
\[
y=(y_1,\ldots,y_m)\in\mathbb R^m
\]
such that
\[
y\cdot v\le y\cdot z
\qquad(v\in L(\mathcal K),\ z\in\mathcal O).
\]
Because $L(\mathcal K)$ is a vector space, replacing $v$ by arbitrary positive and negative multiples shows that
\[
y\cdot v=0\qquad(v\in L(\mathcal K)).
\]
Taking $v=0$ gives $y\cdot z\ge0$ for every $z\in\mathcal O$.  Hence each $y_\ell\ge0$; otherwise one could make the corresponding positive coordinate of $z$ arbitrarily large and force $y\cdot z<0$.

Put the coordinate $y_\ell$ back into the corresponding zero position $e_\ell=\{i_\ell,j_\ell\}$ of $B$: define
\[
X_{i_\ell j_\ell}=X_{j_\ell i_\ell}=y_\ell,\qquad \ell=1,\ldots,m,
\]
and set all remaining entries of $X$, including the diagonal, equal to zero.  Then $X$ is symmetric, nonzero, entrywise nonnegative, and supported on the off-diagonal zeros of $B$.  In particular,
\[
X_{ii}=0,\qquad B_{ij}X_{ij}=0\quad(i\ne j).
\]
For every skew-symmetric $K$,
\begin{align*}
0
&=2y\cdot L(K)\\
&=\tr\!\bigl(X(BK-KB)\bigr)\\
&=\tr\!\bigl(K(XB-BX)\bigr).
\end{align*}
Set $C_0=XB-BX$.  Since $B$ and $X$ are symmetric, $C_0$ is skew-symmetric.  Taking $K=-C_0$ gives
\[
0=-\tr(C_0^2)=\tr(C_0\T C_0),
\]
so $C_0=0$.  Therefore $XB=BX$, completing Part 1.

\medskip
\noindent\textbf{Part 2: reduction of the support graph.}
Let $H$ be the positive off-diagonal support graph of $X$. We aim to show that $H$ is isomorphic to either $K_2\cup P_3$ or $C_5$.  

Let $p \in\mathbb{R}^5$ be the Perron eigenvector of the irreducible nonnegative matrix $B$.  Since $BX=XB$, for any eigenvector $v$ of $B$, $Xv$ is also an eigenvector of $B$ associated with the same eigenvalue. Consequently, when this eigenvalue has a multiplicity of $1$, $Xv\propto v$. Note that the Perron eigenspace of $B$ is one-dimensional. We immediately obtain that   
\begin{equation} \label{temp-equation}
Xp=\nu p, 
\end{equation}
for some real $\nu$.  Because $X\ge0$, $X\ne0$, and $p>0$, we have $\nu>0$. 
By definition, $\nu$ is no larger than the spectral radius of $X$. Meanwhile, consider $Y=D^{-1}XD$, where $D=\diag(p)$. Since $X$ and $Y$ are similar, they have the same spectral radius. Note that $Y{\bf 1}_5=D^{-1}XD{\bf 1}_5 = D^{-1}Xp=\nu D^{-1}p = \nu {\bf 1}_5$. This means that each row of $Y$ has a sum $\nu$. Since $Y$ is nonnegative, $\|Y\|_\infty=\nu$. As a result, the spectral radius of $Y$ is no larger than $\nu$. Combining the above arguments gives that
\begin{equation} \label{spectral-radius}
\mbox{$\nu$ is equal to the spectral radius of $X$}. 
\end{equation}

On each connected component of $H$, the principal submatrix of $X$ on that component is irreducible and has the corresponding positive part of $p$ as an eigenvector with eigenvalue $\nu$.  Hence every component has spectral radius $\nu$, and in particular there are no isolated vertices (because $X$ has zero diagonal entries).

If $H$ is disconnected, its component sizes must be $2+3$.  The two-vertex component is $K_2$.  On three vertices, a connected graph is either the path $P_3$ or the triangle $K_3$.  Thus
\[
H\cong K_2\cup P_3
\quad\text{or}\quad
H\cong K_2\cup K_3.
\]
Here the union means that the two displayed graphs are the two disconnected components.

We exclude $K_2\cup K_3$.  Relabel the vertices so that the $K_2$ component is on $\{1,2\}$ and the $K_3$ component is on $\{3,4,5\}$.  The principal $2\times2$ submatrix of $X$ on $\{1,2\}$ has one positive off-diagonal entry and has eigenvalues $\nu$ and $-\nu$.  The principal $3\times3$ submatrix on $\{3,4,5\}$ has all three off-diagonal entries positive.  It is primitive, so its Perron eigenvalue is $\nu$, and its other two eigenvalues have modulus strictly smaller than $\nu$.  Consequently, the two principal submatrices of $X$ have exactly one eigenvalue in common, namely $\nu$.

Because every edge of $H$ is a zero position of $B$, the principal $2\times2$ block of $B$ on $\{1,2\}$ is diagonal.  Commutation with the corresponding principal block of $X$ forces its two diagonal entries to be equal, so it is $dI_2$ for some $d>0$.  Similarly, the principal $3\times3$ block of $B$ on $\{3,4,5\}$ is diagonal, and commutation with the positive triangle block of $X$ forces all three diagonal entries to be equal, so it is $hI_3$ for some $h>0$.

At this point the $2\times3$ cross block of $B$ is still written in the standard Euclidean basis, and it need not be sparse.  Denote it by $C$. Let $U$ and $V$ be the corresponding $2\times2$ and $3\times3$ orthogonal matrices consisting of the eigenvectors of the two principal blocks of $X$, respectively. Define the $5\times 5$ orthogonal matrix $Q=\diag(U,V)$. Then, 
\begin{equation} \label{B-and-tB}
B = \begin{bmatrix}
dI_2 & C\\
C\T & h I_3
\end{bmatrix}, \qquad \widetilde{B}:= Q\T BQ = \begin{bmatrix}
dI_2 & \widetilde{C}\\
\widetilde{C}\T & h I_3
\end{bmatrix},
\end{equation}
where $\widetilde C=U\T C V$. Then, $\widetilde{B}$ is the representation of $B$ in the new orthonormal eigenbasis
defined by $Q$. Note that $\widetilde{B}$ and $B$ have the same eigenvalues. 

%Choose orthonormal eigenbases for the two principal blocks of $X$, and let $U$ and $V$ be the corresponding $2\times2$ and $3\times3$ orthogonal change-of-basis matrices.  In these eigenbases the cross block becomes $\widetilde C=U\T C V$.

We show that $\widetilde{C}$ must be sparse. 
Note that $u$ and $v$ are eigenvectors of $X$ with eigenvalues $\alpha$ and $\beta$, respectively, then
\[
\beta\,u\T Bv=u\T BXv=u\T X Bv=\alpha\,u\T Bv.
\]
Hence
\[
(\beta-\alpha)u\T Bv=0.
\]
Thus an off-diagonal entry of $\widetilde{B}$ can be nonzero only when the two corresponding eigenvalues of $X$ are equal.  
Since the two principal blocks of $X$ share only the eigenvalue $\nu$, the cross block $\widetilde C$ has only one possibly nonzero entry, corresponding to the two Perron eigenvectors of these two blocks. Without loss of generality, suppose in both $U$ and $V$ the first column corresponds to the Perron eigenvector. Then, $\widetilde{C}$ has the form:
\[
\widetilde C=
\begin{pmatrix}
 t&0&0\\
 0&0&0
\end{pmatrix},
\]
for some real $t$. If $t=0$, then $C=U\widetilde{C}V\T$ is a zero matrix, and $B$ is diagonal. This contradicts the irreducibility of $B$. Hence, $t\neq 0$.
%The irreducibility of $\widetilde{B}$ forces $t\ne0$; otherwise the cross block would vanish and $B$ would split into two blocks.
%Therefore, in the combined orthonormal eigenbasis of $X$, and after placing the two Perron eigenvectors first, the full matrix $B$ has the form

We plug $\widetilde{C}$ into \eqref{B-and-tB} and permute the rows and columns of $\widetilde{B}$ simultaneously so that $t$ becomes the $(1,2)$th entry. Denote by $\widetilde{B}_{\text{perm}}$ the resultant matrix. Then,
\[
\widetilde{B}_{\text{perm}} = 
\begin{pmatrix}
 d&t&0&0&0\\
 t&h&0&0&0\\
 0&0&d&0&0\\
 0&0&0&h&0\\
 0&0&0&0&h
\end{pmatrix}.
\]
%This displayed matrix is the representation of $B$ in the new orthonormal eigenbasis, not in the original standard basis.  
Let $r_+\ge r_-$ be the two eigenvalues of the upper-left $2\times2$ block of $\widetilde{B}_{\text{perm}}$.  Then, 
\[
\Spec(B)=\Spec(\widetilde{B}_{\text{perm}})=\{r_+,r_-,d,h,h\},
\qquad
\mbox{where}\quad r_++r_-=d+h.
\]
But
\[
\tau=\tr(B)=2d+3h.
\]
Since $d,h>0$, each of $d$, $h$, and $r_++r_-=d+h$ is strictly less than $\tau$.  Therefore, $r_+$ is the only possible eigenvalue of $B$ that is larger than or equal to $\tau$. However, this contradicts 
\[
\theta_1\ge\theta_2\ge\theta_3=\tau.
\]
Hence the only disconnected possibility is $K_2\cup P_3$.

Now assume $H$ is connected.  Since $B_{ij}X_{ij}=0$ and $B$ is irreducible, the positive off-diagonal support graph of $B$ is a connected spanning subgraph of the complement $\overline H$.  Hence $\overline H$ is connected as well. We call such $H$ a connected and co-connected graph. Let $m=|E(H)|$. Then, $|E(\overline H)|=10-m$. The connectivity of both $H$ and $\overline H$ implies that $m\geq 4$ and $10-m\geq 4$, which gives 
\[
4\le m\le6.
\]
Up to graph isomorphism, there are only 8 different connected co-connected graphs on 5 vertices (e.g., see \cite{kobata2016enumeration}). 
These eight graphs include the cycle $C_5$ and seven non-cycle graphs.  We now exclude those seven cases.

\begin{figure}[tb!]
\centering
\includegraphics[width=.7\textwidth]{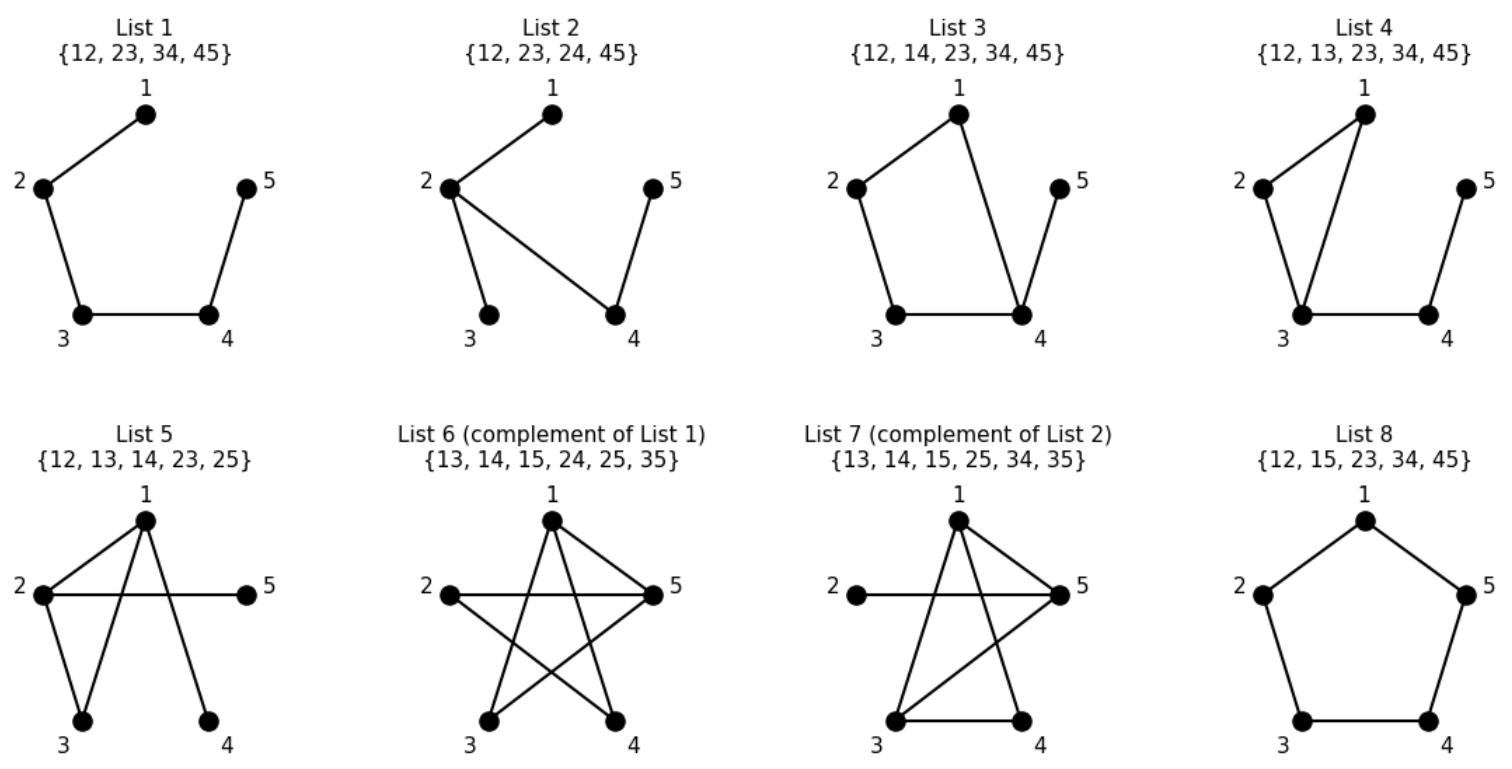}
\caption{Eight connected co-connected graphs.}
\end{figure}

First, consider the case in which $H$ is the path $E(H)=\{12,23,34,45\}$. It follows that
\[
X = 
\begin{pmatrix}
0&a&0&0&0\\
 a&0&b&0&0\\
 0&b&0&c&0\\
 0&0&c&0&d\\
 0&0&0&d&0
\end{pmatrix},
\]
for some positive $a,b,c,d$.  The spectrum of $X$ is simple,  i.e., each eigenvalue has a multiplicity of $1$. To see this,  for a fixed eigenvalue $\lambda$, $Xv=\lambda v$ defines a linear equation for $v=(v_1, v_2,\ldots, v_5)\T$. The form of $X$ ensures that once $v_1$ is given, then $v_2, \ldots, v_5$ can be solved from $v_1$ explicitly. Moreover, $v_1$ must be nonzero (otherwise, $v_2,\ldots,v_5$ will also be zero). This implies that $v$ has only one degree of freedom, i.e., the eigenspace associated with $\lambda$ is one-dimensional. 

 In deriving \eqref{temp-equation}, we have seen that the commutation relationship implies that each eigenvalue of $X$ associated with an eigenvalue with multiplicity $1$ must also be an eigenvector of $B$. Since all eigenvalues of $X$ have a multiplicity of $1$, the eigenvectors of $X$ are the same as those of $B$. Let $Q$ be the $5\times 5$ orthogonal matrix containing the eigenvectors of $X$ and $B$. Then, 
\[
B = Q\begin{pmatrix}\theta_1\\ & \ddots\\ && \theta_5\end{pmatrix} Q\T,  \qquad X = Q\begin{pmatrix}\mu_1\\ & \ddots\\ && \mu_5\end{pmatrix} Q\T, 
\]
where $\mu_1, \ldots, \mu_5$ are distinct. Using the Lagrange interpolation, there exists an order-4 polynomial $f(x)=\sum_{j=0}^4c_jx^j$ such that $\theta_i=f(\mu_i)$, for $i=1,2,\ldots,5$. It follows that
\[
B = f(X)=c_0I+c_1X+c_2X^2+c_3X^3+c_4X^4.
\]
If we view $X$ as the weighted adjacency matrix of a graph, then this graph is bipartite: edges only exist between two classes of nodes, $\{1,3,5\}$ and $\{2,4\}$. For each integer $k$, $(X^k)_{ij}$ is the edge sums of all length-$k$ walks from vertex $i$ to vertex $j$. When $i$ and $j$ belong to distinct classes, there exist no even-length walks between them. Hence, $(X^k)_{12}=(X^k)_{23}=0$, for even $k$. In addition, by direct calculations, 
\[
(X^3)_{12}=a(a^2+b^2) >0,
\qquad
(X^3)_{23}=b(a^2+b^2+c^2) >0.
\]
By the definition of the graph $H$, $X_{12}\neq 0$ and $X_{23}\neq 0$.
This further implies $B_{12}=B_{23}=0$. We plug the above results into $f(X)=B$ to obtain
\[
c_1a+c_3 a(a^2+b^2)=0,
\qquad
c_1b+c_3b(a^2+b^2+c^2)=0.
\]
Since $c\neq 0$, we immediately have $c_3=c_1=0$.  Therefore, $B=c_0I + c_2X^2+c_4X^4$, which is a polynomial of the bipartite matrix $X$ with only even-order terms. Then, for any $i$ and $j$ from the two different node classes, $B_{ij}=0$.  This implies that $B$ is blockwise diagonal with respect to the two blocks $\{1,3,5\}$ and $\{2,4\}$. It contradicts the irreducibility of $B$.

\begin{table}[tb!]
\renewcommand{\arraystretch}{1.25}
\caption{The calculations needed to rule out the five non-cycle graphs. In these calculations, we repeatedly use the facts that  $X_{ij}=0$ when $\{i,j\}\notin E(H)$ and $B_{ij}=0$ when $\{i,j\}\in E(H)$.} \label{tb:graphs-to-rule-out}
\begin{tabular}{p{0.32\textwidth}p{0.65\textwidth}}
\toprule
$E(H)$ & Consequence of $BX=XB$ \\
\midrule
$\{12,23,34,14,45\}$  
& Write $a=X_{12}>0$ and $z=B_{25}\ge0$. By calculations, $(BX)_{15}=0$ and $(XB)_{15}=az$. Then, $BX=XB$ implies $z=0$. The support of $B$ is thus contained in $\{13,15,24,35\}$.  Thus the positive off-diagonal support graph of $B$ is disconnected, contradicting the irreducibility of $B$.  \\
$\{12,23,13,34,45\}$ 
& Write $a=X_{12}$, $b=X_{23}$, $c=X_{13}$, $d=X_{34}$, $e=X_{45}$ and $u=B_{14}$, $v=B_{15}$, $w=B_{24}$, $z=B_{25}$, $t=B_{35}$, where $a, b,c,d,e$ are positive and $u,v,w,z,t$ are nonnegative. The equation $(BX)_{35}=(XB)_{35}$ implies $cv+bz=0$, hence $v=z=0$.  The equations $(BX)_{24}=(XB)_{24}$, $(BX)_{15}=(XB)_{15}$, and $(BX)_{25}=(XB)_{25}$ give $au=0$, $ct=0$, $we=0$, respectively; consequently $u=t=w=0$. Then, $B$ is diagonal, contradicting the irreducibility.\\
$\{12,23,13,14,25\}$ & Write $a=X_{12}$, $b=X_{23}$, $c=X_{13}$, $d=X_{14}$, $e=X_{25}$ and $u=B_{15}$, $v=B_{24}$, $w=B_{34}$, $t=B_{35}$, $s=B_{45}$, with positive $a, b,c,d,e$ and nonnegative  $u,v,w,t,s$. The equations $(BX)_{24}=(XB)_{24}$, $(BX)_{35}=(XB)_{35}$, and $(BX)_{45}=(XB)_{45}$ imply that $wb+se=0$, $uc=0$, $ve=du=0$.  Hence $w=s=u=v=0$.  Then, only $B_{35}$ can be nonzero in the upper triangle (excluding the diagonal), which contradicts the irreducibility of $B$.\\
Complement of $\{12,23,34,45\}$ & The support of $B$ is contained in the path $\{12,23,34,45\}$, so the irreducibility implies $B_{23}>0$. Moreover, $(BX)_{12}=0$. Since $\{13\}$ is in the complement of this path, $X_{13}>0$.  It follows that $(XB)_{12}\ge X_{13}B_{32}>0$, contradicting $(BX)_{12}=0$. \\
Complement of $\{12,23,24,45\}$ & The support of $B$ is contained in the tree $\{12,23,24,45\}$, so the irreducibility makes all four corresponding entries of $B$ positive. Moreover, $(BX)_{12}=0$. Since $\{13,14\}$ are contained in the complement of this tree, $X_{13}>0$ and $X_{14}>0$. It follows that $(XB)_{12}\ge X_{13}B_{32}+X_{14}B_{42}>0$, contradicting $(BX)_{12}=0$.\\
\bottomrule
\end{tabular}
\end{table}

Next,  consider the tree $T_5$ with $E(H)=\{12,23,24,45\}$. Then, 
\[
X = 
\begin{pmatrix}
0&a&0&0&0\\
 a&0&b&c&0\\
 0&b&0&0&0\\
 0&c&0&0&d\\
 0&0&0&d&0
\end{pmatrix},
\]
for positive weights $a,b,c,d$. Under the bipartition $\{2,5\}\cup\{1,3,4\}$, $X$ has four blocks, where the two diagonal blocks are zero, and the off-diagonal block of $X$ satisfies that 
\[
R=\begin{pmatrix}a&b&c\\0&0&d\end{pmatrix},
\qquad
RR\T=
\begin{pmatrix}
 a^2+b^2+c^2&cd\\
 cd&d^2
\end{pmatrix}.
\]
Then, $XX\T$ is block-wise diagonal, where the blocks corresponding to $\{2,5\}$ and $\{1,3,4\}$ are $RR\T$ and $R\T R$, respectively. Let $\alpha\geq 0$ and $\beta\geq 0$ be the two eigenvalues of $RR\T$, which are also the eigenvalues of $R\T R$. Then, the five eigenvalues of $X$ are $\{\pm \sqrt{\alpha}, \pm\sqrt{\beta}, 0\}$.
Note that the $2\times2$ matrix $RR\T$ is positive definite because its determinant is $d^2(a^2+b^2)>0$, and its two eigenvalues are distinct because $cd>0$ (if the two eigenvalues are equal, then the inverse of $RR\T$  is proportional to its transpose; using the inversion formula of $2\times 2$ matrices, this is impossible when $cd>0$).  Thus $X$ has five distinct eigenvalues, and so the spectrum of $X$ is simple. Similarly as in the first case, $B$ can be expressed as a polynomial of $X$ with degree at most 4:  $B=c_0I+c_1X+c_2X^2+c_3X^3+c_4X^4$. The graph $H$ is also bipartite with respect to $\{2,5\}$ and $\{1,3,4\}$. Then, $(X^k)_{12}=(X^k)_{24}=0$ for even $k$. We also observe that 
\[
X_{12}=a, \qquad (X^3)_{12}=a(a^2+b^2+c^2),
\qquad X_{24}=c,\qquad 
(X^3)_{24}=c(a^2+b^2+c^2+d^2).
\]
Since $B_{12}=B_{24}=0$, we again obtain $c_1=c_3=0$, so $B$ is a polynomial of $X$ with only even-order terms. Using similar arguments as before, $B$ is reducible. This yields a contradiction.

The remaining five non-cycle connected support graphs are ruled out directly from $BX=XB$.  The needed computations are summarized in Table~\ref{tb:graphs-to-rule-out}; every displayed variable corresponding to an edge of $X$ is positive, and every displayed entry of $B$ is nonnegative.

Therefore the only connected possibility is $C_5$.  This proves the support reduction.

\medskip

\noindent\textbf{Part 3: the case of $\theta_2>\theta_3$.}
We aim to show that $H\cong K_2\cup P_3$ is impossible when 
\begin{equation}\label{eq:strict-sign-B-copy}
\theta_1>\theta_2>\theta_3=\tr(B)>0>\theta_4\ge\theta_5, \qquad \theta_1\leq 2\tr(B).
\end{equation}
Hence, $H\cong C_5$ is the only possibility. 

Suppose $H\cong K_2\cup P_3$.  Without loss of generality, we assume that $K_2$ contains the nodes $4$ and $5$, and $P_3$ contains the remaining three nodes. Write
\[
a= X_{12} >0,
\qquad
b = X_{23}>0,
\qquad
c = X_{45} >0.
\]
Then, $X$ is blockwise diagonal. By elementary calculations, the spectral radius of the bottom right $2\times 2$ block is equal to $c$, and the spectral radius of the top left $3\times 3$ block is equal to $\sqrt{a^2+b^2}$. 
Meanwhile, in the text below \eqref{spectral-radius}, we have argued that the two blocks share the same spectral radius $\nu$. It follows that 
\[
\nu=c=\sqrt{a^2+b^2}.
\]
Set $
\alpha=a/\nu$ and $
\beta=b/\nu$. Then, $
\alpha^2+\beta^2=1$. 
The top left $3\times 3$ principal submatrix of $X$ has three eigenvalues $\nu,0,-\nu$, with the corresponding eigenvectors as
\begin{equation} \label{u}
u_+=\frac1{\sqrt2}(\alpha,1,\beta)\T,
\qquad
u_0=(\beta,0,-\alpha)\T,
\qquad
u_-=\frac1{\sqrt2}(-\alpha,1,-\beta)\T. 
\end{equation}
For the bottom $2\times 2$ principal sub-matrix of $X$, it has two eigenvalues $\pm \nu$, with the corresponding eigenvectors
\begin{equation} \label{v}
v_+=\frac1{\sqrt2}(1,1)\T,
\qquad
v_-=\frac1{\sqrt2}(1,-1)\T. 
\end{equation}
We can construct $5$-dimensional vectors from $u_{\pm}, u_0$ by adding zeros to the last two coordinates and denote them by $\tilde{u}_\pm, \tilde{u}_0$. Similarly, we construct $5$-dimensional vectors from $v_\pm$ by adding zeros to the first three coordinates and still denote them by $\tilde{v}_\pm$. Now, $\{\tilde{u}_{\pm}, \tilde{u}_0, \tilde{v}_{\pm}\}$ are the eigenvectors of $X$, and the corresponding eigenvalues take three distinct values. Therefore, $X$ has three eigenspaces:
\[
E_+=\operatorname{span}(\tilde{u}_+,\tilde{v}_+),
\qquad
E_0=\operatorname{span}(\tilde{u}_0),
\qquad
E_-=\operatorname{span}(\tilde{u}_-,\tilde{v}_-).
\]

Commutation between $B$ and $X$ implies that each of the above three eigenspaces of $X$ is invariant under $B$. The restriction of $B$ on $E_+$ is a matrix $B_+\in\mathbb{R}^{2\times 2}$ such that $B[\tilde{u}_+, \tilde{v}_+]=[\tilde{u}_+, \tilde{v}_+]B_+$, and the restrictions of $B$ on $E_-$ and $E_0$ are defined similarly. We parameterize these restrictions by
\[
B_+=\begin{pmatrix}d_+&q\\q&e_+\end{pmatrix},
\qquad
B_0 = [h],
\qquad
B_-=\begin{pmatrix}d_-&t\\t&e_-\end{pmatrix}.
\]
Let $A\in\mathbb{R}^{3\times 3}$ and $C\in\mathbb{R}^{2\times 2}$ be the two diagonal blocks of $B$ and let $R\in\mathbb{R}^{3\times 2}$ be the off-diagonal blocks. Then, $B[\tilde{u}_+, \tilde{v}_+]=[u_+, v_+]B_+$ becomes
\begin{equation} \label{restriction-equation}
\begin{bmatrix}
A & R\\ R\T & C
\end{bmatrix}
\begin{bmatrix}
u_+ & {\bf 0}\\ {\bf 0} & v_+
\end{bmatrix} = \begin{bmatrix}
u_+ & {\bf 0}\\ {\bf 0} & v_+
\end{bmatrix}\begin{pmatrix}d_+&q\\q&e_+\end{pmatrix}. 
\end{equation}
It gives $Au_+=d_+u_+$ and $Cv_+=e_+v_+$. We can similarly deduce that $Au_-=d_-u_-$, $Cv_-=e_-v_-$, and $Au_0=hu_0$. Therefore, $u_{\pm}$ and $u_0$ are eigenvectors of $A$, corresponding to eigenvalues $d_{\pm}$ and $h$; and $v_{\pm}$ are eigenvectors of $C$, corresponding to eigenvalues of $e_{\pm}$. Since $\{u_{\pm}, u_0\}$ already form an orthonormal basis, we immediately have:
\[
A = d_+u_+u\T_+ + hu_0u\T_0 + d_-u_-u\T_-.
\]
The support graph of $A$ is contained in the complement of the path $P_3$. Hence, $A_{12}=0$. Using \eqref{u}, we have $A_{12}=\frac{\alpha}{2}(d_+-d_-)$. Since $\alpha>0$, it implies that $d_+=d_-$. Similarly, $C=e_+ v_+ v_+\T+ e_- v_- v_-\T$, and we can use $C_{12}=0$ to show that $e_+=e_-$. It follows that
\[
B_+=\begin{pmatrix}d &q\\q&e\end{pmatrix},
\qquad
B_0 = [h],
\qquad
B_-=\begin{pmatrix}d&t\\t&e\end{pmatrix}, 
\]
and the two diagonal blocks of $B$ are equal to 
\[
A = d (u_+u\T_+ +  u_-u\T_-) + hu_0u\T_0=dI_3 + (h-d)u_0u_0\T, \qquad 
C = e (v_+ v_+\T+ v_- v_-\T)= eI_2.
\]
In particular, $A_{13}=\alpha \beta(d-h)$, so non-negativity gives 
\begin{equation}\label{eq:dgeqh}
d\ge h.
\end{equation}

For the off-diagonal block $R$, the equation \eqref{restriction-equation} implies that $
Rv_+=qu_+$ and $R\T u_+ = qv_+$. As a result, $u_+$ is an eigenvector of $RR\T$, and $v_+$ is an eigenvector of $R\T R$, and their corresponding eigenvalues are both $q^2$. Similarly, $u_-$ and $v_-$ are the eigenvectors of $RR\T$ and $R\T R$, respectively, associated with the eigenvalue $t^2$.  Since $\{u_+,u_0,u_-\}$ is an orthonormal basis of $\mathbb R^3$, $\{v_+,v_-\}$ is an orthonormal basis of $\mathbb R^2$, and invariance of $E_0$ gives $R\T u_0=0$, we obtain 
\[
R = q u_+v_+\T + t u_-v_-\T = \frac{1}{2}\begin{pmatrix} \alpha(q-t) & \alpha(q+t)\\ q+t & q-t \\ \beta(q-t) & \beta(q+t)\end{pmatrix}.
\]
Entrywise non-negativity of $R$ yields $q\ge|t|$. Note that the eigenvalues of $B_+$ are equal to the two eigenvalues associated with the two eigenvectors in the span of $\tilde{u}_+$ and $\tilde{v}_+$, and the eigenvalues of $B_-$ are equal to the two eigenvalues associated with the two eigenvectors in the span of $\tilde{u}_-$ and $\tilde{v}_-$. The Perron eigenvalue $\theta_1$ must appear in $B_+$. 
 
However, when $q=|t|$,  $B_+$ and $B_-$ have the same characteristic polynomial, so that $\theta_1$ will also appear as an eigenvalue of $B_-$. This contradicts the fact that the multiplicity of $\theta_1$ must be $1$. Therefore, 
\[
q>|t|.
\]
Let $p_+>p_-$ be the eigenvalues of $B_+$ and $m_+\ge m_-$ the eigenvalues of $B_-$.  Note that $B_{\pm}$ share the same diagonal entries, and the off-diagonal entries satisfy that $q>|t|$. 
By the explicit formula for eigenvalues of $2\times2$ symmetric matrices, we obtain that 
\begin{equation}\label{eq:block-order}
p_+>m_+\ge m_->p_-.
\end{equation}
These four values, together with $h$, constitute all eigenvalues of $B$.
Note that $p_+=\theta_1$.

If $p_-\ge0$, then $B$ has at least four nonnegative eigenvalues, contradicting \eqref{eq:strict-sign-B-copy}.  Hence $p_-<0$.

If $p_-=\theta_4>\theta_5$, then \eqref{eq:block-order} forces
\[
\{m_+,m_-\}=\{\theta_2,\theta_3\},
\qquad h=\theta_5.
\]
Since $B_+$ and $B_-$ have the same trace $d+e$,
\[
\theta_1+\theta_4=\theta_2+\theta_3.
\]
But $\theta_1\le2\theta_3$ and $\theta_2>\theta_3$, so
\[
\theta_4=\theta_2+\theta_3-\theta_1>0,
\]
which gives a contradiction to \eqref{eq:strict-sign-B-copy}.

Thus $p_-=\theta_5$ (including the case $\theta_4=\theta_5$).  If
\[
\{m_+,m_-\}=\{\theta_2,\theta_3\}, \qquad h = \theta_4, 
\]
then the trace equality gives $\theta_2+\theta_3=\theta_1+\theta_5$. Again, since $\theta_1\le2\theta_3$ and $\theta_2>\theta_3$, we have
\[
\theta_5=\theta_2+\theta_3-\theta_1>0. 
\]
 This contradicts $\theta_5<0$ in \eqref{eq:strict-sign-B-copy}.  If
\[
\{m_+,m_-\}=\{\theta_2,\theta_4\},
\qquad h=\theta_3,
\]
then the trace equality gives $\theta_2+\theta_4=\theta_1+\theta_5$. Moreover, $\tr(B)=\theta_3$  in \eqref{eq:strict-sign-B-copy} gives $\theta_1+\theta_2+\theta_4+\theta_5=0$. Combining these two equalities gives 
\[
\theta_1+\theta_5=0, 
\]
contradicting primitivity of $B$ (because a primitive nonnegative matrix has no other eigenvalue of modulus equal to its Perron root).

The only remaining possibility is
\[
\{m_+,m_-\}=\{\theta_3,\theta_4\},
\qquad h=\theta_2.
\]
But $d$ is a diagonal entry of $B_-$ and therefore lies between its two eigenvalues $\theta_3$ and $\theta_4$, so $d\le\theta_3$.  This contradicts \eqref{eq:dgeqh}, because $h=\theta_2>\theta_3$.  Thus $H\cong K_2\cup P_3$ is impossible.

\medskip
\noindent\textbf{Part 4: the case of $\theta_2=\theta_3$.}
We aim to show that $H\cong K_2\cup P_3$ is impossible when  
\begin{equation} \label{eq:repeated-sign-B-copy}
\theta_1>\theta_2=\theta_3=\tr(B)>0>\theta_4\ge\theta_5, \qquad \theta_1\leq 2\tr(B). 
\end{equation}
%Write
%\[
%\eta=\theta_4,\qquad \zeta=\theta_5,
%\]
%so that $0>\eta\ge\zeta$.

Suppose $H\cong K_2\cup P_3$. The arguments in Part 3 are still valid up until \eqref{eq:block-order}.
Now, if $p_-\geq 0$, then $B$ has at least four nonnegative eigenvalues, which  contradicts \eqref{eq:repeated-sign-B-copy}. Hence, $p_-<0$. This means that only $p_-=\theta_4$ or $p_-=\theta_5$ are possible. 

If $p_-=\theta_4>\theta_5$, then $\{m_+, m_-\}=\{\theta_2, \theta_3\}$. Under \eqref{eq:repeated-sign-B-copy}, it forces
\[
m_+= m_-=\theta_3, \qquad h = \theta_5. 
\]
Note that $B_+$ and $B_-$ have the same trace. Hence, $2\theta_3 = \theta_1 + \theta_4$. Since $\theta_1\leq 2\tr(B)=2\theta_3$, it must hold that $\theta_4\geq 0$, which contradicts \eqref{eq:repeated-sign-B-copy}. 
Thus, $p_-=\theta_5$ is the only possible case. Then, $m_+, m_-, h$ take the three values $\theta_3, \theta_3, \theta_4$. If 
\[
m_+=m_-=\theta_3,  \qquad h = \theta_4, 
\]
then the trace equality implies $2\tr(B)=\theta_1+\theta_5$. This contradicts $\theta_1\leq 2\tr(B)=2\theta_3$ and $\theta_5<0$ in \eqref{eq:repeated-sign-B-copy}. If
\[
m_+=\theta_3, \qquad m_-=\theta_4, \qquad h = \theta_3. 
\]
The trace equality gives $\theta_3+\theta_4=\theta_1+\theta_5$. Meanwhile, $\tr(B)=\theta_3$ yields that $\theta_1+\theta_3+\theta_4+\theta_5=0$. We combine them to obtain
\[
\theta_1+\theta_5=0, 
\]
contradicting primitivity of $B$. We have ruled out all possibilities, indicating that $H\cong K_2\cup P_3$ is impossible. 
\end{proof}

\subsection{Proof of Lemma~\ref{lem:cycle-rep}}
We use the cyclic relabeling for both $B$ and $X$. From now on, write $B=(B_{i,i+k})$ for $i\in \mathbb Z/5\mathbb Z$ and $k\in \{-2, -1, 0, 1, 2\}$. Here, the indices should be read modulo $5$, i.e., $B_{1,6}=B_{1,1}$ and $B_{3,7}=B_{3,2}$. In addition, by symmetry, $B_{i, i+3}=B_{i, i-2}=B_{i-2,i}$ and $B_{i, i+4}=B_{i, i-1}=B_{i-1,i}$. As a result, to describe the full matrix $B$ (or $X$), we only need to specify 
\[
B_{i, i}, B_{i, i+1}, B_{i, i+2} \mbox{ (or $X_{i,i}, X_{i, i+1}, X_{i, i+2}$)} \qquad\mbox{for }i\in \mathbb Z/5\mathbb Z. 
\]
Note that $X$ has zero diagonal and is supported on $C_5$. Hence, $X_{i, i}=X_{i, i+2}= 0$ and $X_{i,i+1}>0$ under the cyclic relabeling. Furthermore, because $X_{i,i+1}B_{i,i+1}=0$, we have $B_{i, i+1}=0$ under the cyclic relabeling. Given these observations, we let
\[
X_{i,i+1}=u_i>0, 
\qquad B_{i,i}=d_i>0, \qquad B_{i, i+2}=v_i\geq 0, \qquad 
\mbox{for}\quad i\in\mathbb Z/5\mathbb Z.
\]

We first use the $(i,i+2)$ entries of $BX=XB$. Using the facts that  $X_{i,i+k}=0$ for $k\in\{-2,0,2\}$  and $B_{i,i+1}=B_{i,i-1}=0$, we have the following calculations:
\begin{align*}
(BX)_{i,i+2}& =\sum_{k=-2}^2 B_{i, i+k}X_{i+k,i+2}= B_{i,i-2}X_{i-2,i+2}=B_{i, i-2}X_{i+3,i+2}=v_{i-2}u_{i+2},\cr
(XB)_{i,i+2}& =\sum_{k=-2}^2 X_{i, i+k}B_{i+k,i+2}= X_{i,i-1}B_{i-1,i+2}=X_{i,i-1}B_{i+4,i+2} = u_{i-1}v_{i+2}.
\end{align*}
It yields  $v_{i-2}u_{i+2}=u_{i-1}v_{i+2}$. Since each $u_j>0$, dividing both sides  by $u_{i-2}u_{i-1}u_{i+2}$ and utilizing $u_{i-2}=u_{i+3}$ (because the indices are read modulo 5)  gives 
\[
\frac{v_{i-2}}{u_{i-2}u_{i-1}}
=
\frac{v_{i+2}}{u_{i+2}u_{i+3}}.
\]
As $i$ varies modulo $5$, these relations show that the value of $\frac{v_i}{u_i u_{i+1}}$ does not depend on $i$. Let $b\geq 0$ denote this value. Then, 
\begin{equation}\label{eq:vi-b}
v_i=b u_i u_{i+1}, \qquad \mbox{for all }i\in\mathbb Z/5\mathbb Z. 
\end{equation}

Next use the $(i,i+1)$ entries of $BX=XB$.  By direct calculations, 
\begin{align*}
(BX)_{i,i+1}& =B_{i,i}X_{i,i+1}+ B_{i,i+2}X_{i+2,i+1} =d_i u_i + v_i u_{i+1},\cr
(XB)_{i,i+1}& = X_{i,i-1}B_{i-1,i+1} + X_{i, i+1}B_{i+1, i+2} = u_{i-1}v_{i-1}+u_i d_{i+1}.
\end{align*}
This  gives $d_i u_i+v_i u_{i+1}=u_{i-1}v_{i-1}+u_i d_{i+1}$. 
Substituting \eqref{eq:vi-b} and dividing by $u_i>0$ yields
\[
d_i+b u_{i+1}^2
=d_{i+1}+b u_{i-1}^2.
\]
Equivalently,
\[
d_i-b(u_{i-1}^2+u_i^2)
=
d_{i+1}-b(u_i^2+u_{i+1}^2).
\]
Thus there is a constant $a\in\mathbb R$ such that
\begin{equation}\label{eq:di-a}
d_i=a+b(u_{i-1}^2+u_i^2)
\qquad\text{for all } i\in\mathbb Z/5\mathbb Z. 
\end{equation}

We now show the claim. Note that by direct calculations, 
\[
(X^2)_{ii}=u_{i-1}^2+u_i^2,
\qquad
(X^2)_{i,i+2}=u_i u_{i+1},
\qquad
(X^2)_{i,i+1}=0.
\]
Together with \eqref{eq:vi-b} and \eqref{eq:di-a}, these identities give
\[
B=aI_5+bX^2.
\]
If $b=0$, then $B=aI_5$, which is reducible, contradicting the irreducibility assumption on $B$.  Hence $b>0$.
Moreover, $v_i=b u_i u_{i+1}>0$ for every $i$, so all five off-diagonal edges complementary to the support cycle of $X$ occur in the positive off-diagonal support of $B$.  The two positive off-diagonal support graphs are therefore complementary copies of $C_5$. \qed

\subsection{Proof of Lemma~\ref{lem:C5poly}}\label{app:C5poly}
\noindent Consider claim (a).  Write 
\[
tI_5-X=
\begin{pmatrix}
t&-y_1&0&0&-y_5\\
-y_1&t&-y_2&0&0\\
0&-y_2&t&-y_3&0\\
0&0&-y_3&t&-y_4\\
-y_5&0&0&-y_4&t
\end{pmatrix}.
\]
By the Leibniz formula, $\det(tI_5-X)$ is a sum of terms indexed by the
permutations $\sigma\in S_5$, where $S_5$ denotes the set of all
permutations of $\{1,2,3,4,5\}$.  More precisely,
\[
\det(tI_5-X)
=\sum_{\sigma\in S_5}\operatorname{sgn}(\sigma)
\prod_{i=1}^5 (tI_5-X)_{i,\sigma(i)}.
\]
Now, for each $\sigma$, the associated term  can be nonzero only when every selected entry
$(tI_5-X)_{i,\sigma(i)}$ is nonzero.  If $\sigma$ contains a cycle of
length $3$, the three vertices in that cycle would have to form a triangle
in the support graph of $X$, so the corresponding term is zero.  Similarly,
a cycle of length $4$ would require a four-cycle in the support graph and
also gives zero.  Hence it remains to consider only the identity
permutation, one transposition, two disjoint transpositions, and a cycle of
length $5$.

The identity permutation selects all five diagonal entries and contributes
$t^5$.  A nonzero transposition must exchange the endpoints of one of the
five edges.  For the edge with weight $y_i$, its sign is $-1$, the two
off-diagonal entries have product $y_i^2$, and the remaining three
diagonal entries contribute $t^3$.  Summing these five terms gives
$-\left(\sum_{i=1}^5y_i^2\right)t^3$.

For two disjoint transpositions, the two edges must have weights $y_i$
and $y_{i+2}$ for some $i$, with indices modulo $5$.  The sign is $+1$,
the four off-diagonal entries have product $y_i^2y_{i+2}^2$, and the
remaining diagonal entry contributes $t$.  Summing the five possible
choices gives $\left(\sum_{i=1}^5y_i^2y_{i+2}^2\right)t$.

Finally, a nonzero cycle of length $5$ must follow one of the two
orientations of the five-cycle.  Each permutation has sign $+1$, while
the product of the five selected entries is
$-\prod_{i=1}^5y_i$.  The two orientations therefore contribute
$-2\prod_{i=1}^5y_i$.  Combining these contributions gives
\[
\det(tI_5-X)=t^5-\left(\sum_{i=1}^5y_i^2\right)t^3+\left(\sum_{i=1}^5y_i^2y_{i+2}^2\right)t-2\prod_{i=1}^5y_i.
\]
This proves claim (a). 

\medskip
\noindent Consider claim (b).  Setting $t=0$ in claim (a) gives
$p_X(0)=\det(-X)=-2\prod_{i=1}^5y_i$.  Since $X$ has size $5$,
$\det(-X)=-\det(X)$, so $\det(X)=2\prod_{i=1}^5y_i>0$ and $X$ is
nonsingular. Now, delete the fifth row and column of $X$ and write the resulting principal
submatrix as
\[
Y=
\begin{pmatrix}
0&y_1&0&0\\
y_1&0&y_2&0\\
0&y_2&0&y_3\\
0&0&y_3&0
\end{pmatrix}.
\]
Let $\Delta_k$ denote the determinant of the leading $k\times k$ block of
$Y$, with $\Delta_0=1$.  The last row and column of this tridiagonal block
have only the off-diagonal entry $y_{k-1}$, so expansion along them gives
$\Delta_k=-y_{k-1}^2\Delta_{k-2}$ for $k\ge2$.  Thus
$\Delta_1=0$, $\Delta_2=-y_1^2$, $\Delta_3=0$, and
$\det(Y)=\Delta_4=-y_3^2\Delta_2=y_1^2y_3^2>0$; in particular, $Y$ has
no zero eigenvalue.  For $D=\diag(1,-1,1,-1)$, direct multiplication gives
$DYD=-Y$.
Since $D^2=I_4$, this is equivalent to $YD=-DY$.  Thus
$Yv=\nu v$ implies $Y(Dv)=-\nu(Dv)$, so the eigenvalues of $Y$ occur in
opposite pairs.  As $Y$ is real symmetric, they can be ordered as
$\nu_1\ge\nu_2>0>\nu_3\ge\nu_4$.

Order the eigenvalues of $X$ as
$\mu_1\ge\mu_2\ge\mu_3\ge\mu_4\ge\mu_5$.
Cauchy's interlacing theorem for the principal submatrix $Y$ gives
\[
\mu_1\ge\nu_1\ge\mu_2\ge\nu_2\ge\mu_3
\ge\nu_3\ge\mu_4\ge\nu_4\ge\mu_5.
\]
Because $\nu_2>0$ and $\nu_4<0$, this implies
$\mu_1,\mu_2>0$ and $\mu_4,\mu_5<0$.  Since $X$ is nonsingular,
$\mu_3\ne0$, and the positivity of
$\det(X)=\mu_1\mu_2\mu_3\mu_4\mu_5$ forces $\mu_3>0$.  Thus $X$ has
exactly three positive and two negative eigenvalues. This proves claim (b). 

\medskip
\noindent Consider claim (c).  Suppose that $p_X(r)=0$.  Since
$p_X(0)=-2\prod_{i=1}^5y_i\ne0$, we have $r\ne0$.  By
\eqref{eq:C5poly}, all nonconstant terms of $p_X(t)$ are odd powers of
$t$.  Hence
\[
p_X(r)+p_X(-r)=-4\prod_{i=1}^5y_i<0.
\]
Using $p_X(r)=0$, we obtain
$p_X(-r)=-4\prod_{i=1}^5y_i\ne0$.  Thus $-r$ is not an eigenvalue of
$X$. This proves claim  (c).  

\medskip
\noindent Consider claim (d).  Factoring the characteristic polynomial into
linear factors gives
\[
\begin{aligned}
p_X(t)
&=\prod_{j=1}^5(t-\mu_j)\\
&=t^5-\left(\sum_{j=1}^5\mu_j\right)t^4
+\left(\sum_{1\le i<j\le5}\mu_i\mu_j\right)t^3\\
&\quad-\left(\sum_{1\le i<j<k\le5}\mu_i\mu_j\mu_k\right)t^2\\
&\quad+\left(\sum_{1\le i<j<k<\ell\le5}
\mu_i\mu_j\mu_k\mu_\ell\right)t
-\prod_{j=1}^5\mu_j.
\end{aligned}
\]
Comparing coefficients with \eqref{eq:C5poly} gives
\[
c_1(\mu)=0,\qquad
\sum_{1\le i<j\le5}\mu_i\mu_j=-\sum_{i=1}^5y_i^2,\qquad
\sum_{1\le i<j<k\le5}\mu_i\mu_j\mu_k=0,
\]
and \(\prod_{j=1}^5\mu_j=2\prod_{i=1}^5y_i\).  Since
\[
c_1(\mu)^2=c_2(\mu)+2\sum_{1\le i<j\le5}\mu_i\mu_j,
\]
the first two identities give
\(c_2(\mu)=2\sum_{i=1}^5y_i^2\).  Newton's third identity is
\[
c_3(\mu)-c_1(\mu)c_2(\mu)
+c_1(\mu)\sum_{1\le i<j\le5}\mu_i\mu_j
-3\sum_{1\le i<j<k\le5}\mu_i\mu_j\mu_k=0.
\]
Using \(c_1(\mu)=0\) and
\(\sum_{i<j<k}\mu_i\mu_j\mu_k=0\), we obtain \(c_3(\mu)=0\).
This proves claim (d).
\qed

\bibliographystyle{plainnat}
% REVISION: updated to match the supplied bibliography filename references(5).bib.
\bibliography{references}
\end{document}